\documentclass[final,onefignum,onetabnum]{siamart251216}

\usepackage{lipsum}
\usepackage{amsfonts}
\usepackage{graphicx}
\usepackage{multirow}
\usepackage{epstopdf}
\usepackage{algorithmic}
\ifpdf
  \DeclareGraphicsExtensions{.eps,.pdf,.png,.jpg}
\else
  \DeclareGraphicsExtensions{.eps}
\fi

\newsiamremark{remark}{Remark}
\newsiamremark{hypothesis}{Hypothesis}
\crefname{hypothesis}{Hypothesis}{Hypotheses}
\newsiamthm{claim}{Claim}
\newsiamthm{question}{Question}

\headers{Quantum-inspired FEM}{X.~Wang, J.~H.~Adler, and X.~Hu}

\title{Quantum-inspired methods for finite-element discretizations of the high-dimensional Poisson equation\thanks{Submitted to the editors DATE.
\funding{The work of the last two authors is partially supported by the National Science Foundation (NSF) under grant DMS-2513394.}}}

\author{Xue Wang\thanks{School of Mathematics, Shandong University, 250100 Jinan, China
(\email{wangxsdu@mail.sdu.edu.cn}).}
\and
James H. Adler\thanks{Department of Mathematics, Tufts University, Medford, 02155 MA, USA
  (\email{james.adler@tufts.edu}),
  (\email{xiaozhe.hu@tufts.edu}).}
\and Xiaozhe Hu\footnotemark[3]
}

\usepackage{amsopn}

\usepackage{bm}
\usepackage{algorithm}
\usepackage{algorithmic}
\usepackage{graphicx}
\usepackage{subcaption}

\ifpdf
\hypersetup{
  pdftitle={Quantum-inspired FEM},
  pdfauthor={J.~H.~Adler, X. Hu, and X. Wang}
}
\fi

\begin{document}

\maketitle

% REQUIRED
\begin{abstract}
In recent years, quantum linear system algorithms have been applied to partial differential equations (PDEs), particularly in high-dimensional settings, demonstrating an exponential speedup in dimension. Concurrently, randomized and quantum-inspired classical linear solvers have emerged, showing computational complexity comparable to their quantum counterparts in many application areas. In this paper, we investigate the applicability of these quantum-inspired classical algorithms to PDEs. We provide both upper and lower bounds on their computational complexity, proving that these methods cannot achieve exponential speedup in dimension for discretizations of high-dimensional Poisson problems. Our theoretical findings definitively demonstrate that quantum-inspired classical algorithms are not competitive with quantum algorithms for solving PDEs, confirming that quantum methods retain a significant advantage for high-dimensional problems.
\end{abstract}

% REQUIRED

% REQUIRED

\section{Introduction}
When solving a general partial differential equation (PDE) in $d$-dimensions,
\begin{equation} \label{eqn:PDE-model}
	\mathcal{L} u = f, \ \text{in} \ \mathcal{D}\subset\mathbb{R}^d,
\end{equation}
where $\mathcal{L}$ denotes a general elliptic second-order partial differential operator, most discretizations, such as finite-difference, finite-element, and finite-volume methods, lead to a discrete linear system of equations,
\begin{equation} \label{eqn:Ax=b}
	A \bm{u} = \bm{f},
\end{equation}
where $A$ is the so-called stiffness matrix.  Solving~\eqref{eqn:Ax=b} is challenging in practice because it is typically large-scale and ill-conditioned.   In this work, we consider the Laplacian equation, i.e., $\mathcal{L} = - \Delta$, as an example.  In this case, the condition number $\kappa(A)$ grows as $N^{2/d}$ where $N$ is the total number of degrees of freedom used in the discretization.  Thus, in practice, the most time-consuming part of solving PDEs is solving~\eqref{eqn:Ax=b}, especially when $N$ gets larger and larger.  Moreover, one must deal with the well-known \textit{``curse of dimensionality"} as the problem size grows exponentially with respect to the dimension \cite{werschulz1991computational,ritter1996average}.  

Much work has been dedicated to overcoming these issues in the classical computing setting \cite{benzi1996sparse,bramble1990parallel}. One of the most successful approaches for solving~\eqref{eqn:Ax=b} in PDE applications is the multigrid (MG) method.  MG provides a preconditioner $B$ such that $\kappa(BA)$ is well-conditioned, providing an improved scaling of the computational complexity of the problem with $N$.  However, as we highlight later, this still does not provide an improved scaling with dimension $d$.  Therefore, we look to quantum computing algorithms instead. 
%to address the \textit{``curse of dimensionality"}.  

Recently, there has been extensive literature for solving PDEs using quantum algorithms, see \cite{clader2013preconditioned,berry2014high,montanaro2016quantum,childs2017quantum,lubasch2020variational}.  Here, we consider the application of quantum algorithms to the finite-element method (FEM) in particular.  Following \cite{montanaro2016quantum}, the key idea of \emph{quantum FEM} is to replace the classical algorithm for solving the linear system \eqref{eqn:Ax=b} with a quantum linear solver. There are many quantum linear solvers, i.e., \cite{harrow2009quantum, somma2016quantum, childs2017quantum, subacsi2019quantum,an2022quantum,bravo2023variational,jin2025quantum}, including an approach proposed in \cite{childs2017quantum}, which is an improvement of the famous Harrow,
Hassidim, and Lloyd (HHL) algorithm \cite{harrow2009quantum,somma2016quantum}.  For more recent developments of quantum linear solvers, we refer to the following survey paper \cite{x6gh-d8gh}.  As is discussed below, using such a quantum linear solver, and those developed more recently, achieves exponential speedup in terms of the dimension. 

However, while quantum computing hardware has developed rapidly in the past few years \cite{de2021materials,wintersperger2023neutral}, the ability to solve large-scale problems is still limited.  Thus, the question is whether the ideas in the quantum algorithms can be leveraged and still applied on classical machines.  This leads to the notion of \emph{quantum-inspired} algorithms.  Recently, quantum-inspired classical algorithms have been proposed for a variety of applications, which achieve comparable computational complexity compared with their fully quantum counterparts \cite{tang2019quantum,tang2021quantum,chia2022sampling,gilyen2022improved,shao2022faster}. 
In this work, we are interested in answering the question, 
\begin{question}\label{ques1}
Can quantum-inspired algorithms achieve comparable complexity for the FEM in particular and, importantly, achieve exponential speedup in dimension, overcoming the curse of dimensionality?
\end{question}
In this context, we again consider such algorithms for the linear solver portion of the solution method.  To the best of our knowledge, the state-of-the-art quantum-inspired linear solvers for various types of linear systems are proposed and summarized in  \cite{shao2022faster}.  The stiffness matrices we consider here, resulting from FEM discretizations of the Laplace equation, $A\in\mathbb{R}^{N\times N}$, are sparse and symmetric positive definite (SPD).  The randomized coordinate descent (RCD) method proposed in \cite{leventhal2010randomized} is considered to be the best known quantum-inspired algorithm according to \cite{shao2022faster}.
Therefore, in this manuscript, we first investigate the computational complexity of the RCD algorithm applied to \eqref{eqn:Ax=b} in both an unpreconditioned and preconditioned setting.  Unfortunately, by establishing both upper and lower bounds on the complexity, we conclude that it is not possible to achieve exponential speedup in dimension when applied to the FEM. However, one might ask,
 \begin{question}\label{ques2}
Whether there exists a yet-undiscovered quantum-inspired classical algorithm that could achieve the same exponential speedup in dimension as its quantum counterparts when applied to FEM discretizations of elliptic PDEs?
 \end{question}
To address this, in the second part of this work, we consider communication complexity, following the approach in \cite{mande2025lower} for establishing lower bounds on quantum-inspired classical algorithms for general linear systems. More precisely, we analyze the complexity of solving FEM linear systems via quantum-inspired classical algorithms within a sampling and query access model defined in \cite{chia2022sampling}. While the complexity analysis in \cite{mande2025lower} provides a lower bound for general matrices, which inherently applies to systems discretized from PDEs, this general bound may not be tight for the specific subset of FEM linear systems we consider here. Consequently, we provide a lower bound specifically for linear systems discretized from the PDEs described in \eqref{eqn:Ax=b}.  Again, we confirm the negative result, that no such algorithm can overcome the exponential scaling with dimension. 
%All results are summarized in \Cref{tab-introd}.  
Therefore, to truly see polynomial scaling, one must consider a fully quantum algorithm. Finally, we note that while the main focus of this manuscript is on finite-element discretizations, the results hold for most discretization approaches applied to elliptic PDEs.

 The rest of the paper is organized as follows: \Cref{sec:preliminaries} sets up some general notation and recalls the classical and quantum FEM complexity analysis for the Poisson model. \Cref{sec:RCD-FEM} introduces the randomized and quantum-inspired algorithms for solving linear systems, and derives the upper and lower bounds on the computational complexity for both the original system and an expanded system (MG preconditioned) using RCD. Next, in \Cref{sec:QICA-FEM}, we derive a lower bound for general quantum-inspired algorithms based on communication complexity. Concluding remarks are provided in \Cref{sec:conclusions}.

\section{Preliminaries} \label{sec:preliminaries}
In this section, we first recall the classical FEM and its complexity analysis. We then introduce quantum algorithms and their complexity analysis, which demonstrates an exponential speedup in terms of the dimension. Finally, we briefly review the state-of-the-art randomized and quantum-inspired classical algorithms for solving general linear systems.

\subsection{Notation}
Before beginning, we clarify our notation in terms of comparing the orders of terms based on lower and upper bounds. Given functions $f(n)$ and $g(n)$, and constants $c > 0$ and $n_0 > 0$, 
\begin{itemize}
	\item $f(n) = \mathcal{O}(g(n))$, means that $0 \leq f(n) \leq c \cdot g(n)$ for $n\geq n_0$.
	\item $f(n) = {\Omega}(g(n))$, means that $0 \leq c \cdot g(n) \leq f(n)$ for $n\geq n_0$.
	\item $f(n) = {\Theta}(g(n))$ means that there hold $f(n) = \mathcal{O}(g(n))$ and $f(n) = \Omega(g(n))$.
	\item $f(n) = \operatorname{poly}(n) $ means $f(n) = \Theta(n^p)$ for some $p \in \mathbb{R}$.
\end{itemize}

\subsection{Classical FEM and its Complexity Analysis} \label{subsec:CFEM-complexity}
For the remainder of the paper, we focus on the study of the Poisson equation, i.e., $\mathcal{L} = -\Delta$, in arbitrary dimension, $d$. Consider an open, bounded, connected, and sufficiently regular domain $\mathcal{D}\subset\mathbb{R}^d$. The goal is then to solve,
\begin{align}
    -\Delta u = f,&\quad   \text{ in }  \mathcal{D}, \label{eqn:Poisson} \\
    u = 0, &\quad  \text{ on } \Gamma_D, \label{eqn:Poisson-DBC} \\
    \frac{\partial u}{\partial n} = 0, &\quad  \text{ on } \Gamma_N, \label{eqn:Poisson-NBC}
\end{align}
with boundary $\partial\mathcal{D} = \Gamma_D\cup\Gamma_N$, and where $f$ is the given source term. Given a sufficiently smooth ``test function" $v\in V = \{v\in [H^1(\mathcal{D})]^d: v|_{\Gamma_D} = 0\}$, one can multiply both sides by $v$ and integrate by parts to obtain the corresponding weak formulation of \cref{eqn:Poisson}-\cref{eqn:Poisson-NBC}: Find $u\in V$ such that
\begin{align*}
    a(u,v) = (f,v), \quad \forall\, v\in V,
\end{align*}
where the bilinear form is defined as
\begin{align*}
    a(u,v) = \int_{\mathcal{D}} (\nabla u\cdot \nabla v)\mathrm{~d} \bm{x}.
\end{align*}

The FEM then considers the solution and test functions in some finite-dimensional subspaces $V_h \subset V$. Denoting the approximation solution as $u_h$, then the discretization scheme is as follows: Find $u_h\in V_h$ such that,
\begin{align*}
    a(u_h,v_h) = (f,v_h), \quad \forall\, v_h\in V_h. 
\end{align*}
Commonly $V_h$ is taken to be the space of piecewise polynomial functions of some degree $p$. For simplicity, we consider $p=1$, piecewise linear functions.  Then, the approximate solution to Poisson’s equation can be
determined by solving the linear system \eqref{eqn:Ax=b}, where the corresponding matrix is derived from $a(\cdot,
\cdot)\mapsto A$.  We have also assumed an underlying triangulation of the domain, decomposed into elements with mesh spacing $h$.  In particular, assume that we use a quasi-uniform mesh,  which gives a relationship between the total number of degrees of freedom, $N$, the mesh spacing, and the dimension of the problem as $N=\Theta(h^{-d})$. %\textcolor{blue}{James: should we make this more precise?}

%In general, $N$ is the total number of degree and $h$ is the mesh size. For examples,  when $d = 1$, we consider $\mathcal{D} = [0,1]$, divided into $N$ intervals of size $h$.  We will make use of the matrix stencil when solving the linear system, so we can reduce the memory usage. 

%In this section, we compare the computational complexity of RCD and CG for the classical finite element discretization below a specified tolerance. 

Let $u$, $u_h$, and $u_h^k$ denote the exact solution, the exact finite-element solution in $V_h$, and the finite-element approximation obtained by an iterative linear solver, respectively.  Since $u_h$ and $u_h^k$ are finite-element functions, we also denote them by their vector representations, $\bm{u}$ and $\bm{u}^k$, respectively.  Using the triangle inequality, the relative error of the method is then decomposed into two parts,
\begin{equation*}
	\frac{| u - u_h^k |_1}{\|f\|} \leq \overbrace{\frac{| u - u_h |_1}{\|f\|}}^{\text{discretization error}}+ \overbrace{\frac{| u _h - u_h^k|_1}{\|f\|}}^{\text{linear solver error}}.
\end{equation*}
Here, $|\cdot|_{m}$ is the Sobolev seminorm \( |v|_m := (\sum\limits_{\alpha, |\alpha|=m} \|\partial^{\alpha} v\|^2)^{\frac12} \), where \(\alpha = (\alpha_1, \ldots, \alpha_d)\) is a multi-index, \(|\alpha| := \sum_{i=1}^d \alpha_i\), and  
\(
\partial^\alpha := \left( \frac{\partial}{\partial x_1} \right)^{\alpha_1} \cdots \left( \frac{\partial}{\partial x_d} \right)^{\alpha_d}.
\)
The full Sobolev \(m\)-norm is 
\(
\| v \|_{m} := \sum_{i=0}^m |v|_i,
\)
and we note that for \(m = 0\), \(|v|_0 = \| v \|_{0} = \| v \|\) is the $L^2(\mathcal{D})$ norm. 

We would like to produce a solution to meet a desired accuracy
\begin{equation} \label{eqn:target-accuracy}
	\frac{| u - u_h^k |_1}{\|f\|} = \mathcal{O}( \operatorname{poly}(d) \epsilon).
\end{equation} 
Therefore, it is sufficient to require that the discretization error and the linear solver error each satisfy 
\[\displaystyle \frac{| u - u_h |_1}{\|f\|} =\mathcal{O}( \operatorname{poly}(d) \epsilon), \quad \text{and}\quad \displaystyle\frac{| u _h - u_h^k|_1}{\|f\|}= \mathcal{O}( \operatorname{poly}(d) \epsilon).\]
Focusing on the discretization error first, we note the standard $H^1$ error estimates for the FEM with piecewise linear basis \cite{brenner2008mathematical,jiang2025polynomial}, under a full regularity assumption, 
\begin{equation*}
	| u - u_h |_1 \leq C  \operatorname{poly}(d)  h \| u \|_{2} \leq C \operatorname{poly}(d) h\|f\| .
\end{equation*}
Here, $C$ is generic constant independent of $h$ and $d$.   The exiplicit formulation of the factor $\operatorname{poly}(d)$ depends on the shape-regularity and quasi-uniformity of the mesh (we refer to \cite{jiang2025polynomial} for details).  Thus, to achieve the targeted accuracy \eqref{eqn:target-accuracy}, we  require
%\begin{equation*}
%	C h \leq \epsilon.	
%\end{equation*}
%This implies that the mesh size needs to be
\begin{equation}\label{h-epsilon}
	h = \Theta(\epsilon).
\end{equation}
Thus, the discretization error determines the mesh size and, consequently, the size of the linear system:
\begin{equation}\label{eqn:estimate-N}
	N = \Theta(h^{-d}) = \Theta(\epsilon^{-d}).
\end{equation} 

The remaining part of the overall computational complexity is then determined by the cost of solving the linear system to meet the accuracy criterion\newline $\displaystyle\frac{| u _h - u_h^k|_1}{\|f\|} =  \mathcal{O}( \operatorname{poly}(d) \epsilon)$.  Note that $| u_h - u_h^k |_1 = \| \bm{u} - \bm{u}^k \|_A$,  where $A \in \mathbb{R}^{N\times N}$ is the corresponding stiffness matrix and $\|\cdot\|_A = \sqrt{a(\cdot,\cdot)}$ is the energy norm.   Therefore, it is sufficient to require that
\begin{equation}\label{eqn:lin_solve_error}
	\frac{\| \bm{u} - \bm{u}^k \|_A}{\|f\|}  =  \mathcal{O}( \operatorname{poly}(d) \epsilon).
\end{equation}

For the Poisson problem discretized with linear finite elements, a common choice of linear solver is the \textit{conjugate gradient} (CG) method \cite{saad2003iterative}. Assuming an initial guess of $\bm{u}^0 = \bm{0}$ and that $h<1$, standard convergence analysis yields,
\begin{align*}
	\| \bm{u} - \bm{u}^k \|_A &\leq 2 \left( \frac{\sqrt{\kappa(A)} - 1}{ \sqrt{\kappa(A)}+1} \right)^k \| \bm{u} \|_A = 2 \left( \frac{\sqrt{\kappa(A)} - 1}{ \sqrt{\kappa(A)}+1} \right)^k | u_h |_1 \\ 
	&\leq 2 \left( \frac{\sqrt{\kappa(A)} - 1}{ \sqrt{\kappa(A)}+1} \right)^k (| u |_1 +  |u -u_h |_1) \\
	& \leq C \operatorname{poly}(d) \left( \frac{\sqrt{\kappa(A)} - 1}{ \sqrt{\kappa(A)}+1} \right)^k (\| u\|_{2} + h \| u \|_{2}) \\
	%& = 2 (1 + h) \operatorname{poly}(d) \left( \frac{\sqrt{\kappa(A)} - 1}{ \sqrt{\kappa(A)}+1} \right)^k \| u\|_{2}  
	& \leq C \operatorname{poly}(d) \left( \frac{\sqrt{\kappa(A)} - 1}{ \sqrt{\kappa(A)}+1} \right)^k \| f\|.
\end{align*}
Together with~\eqref{eqn:lin_solve_error}, %we have
%\begin{equation*}
%	C \left( \frac{\sqrt{\kappa(A)} - 1}{ \sqrt{\kappa(A)}+1} \right)^k   = \mathcal{O}(\epsilon). 
%\end{equation*} 
we find that the number of CG iterations required to achieve the $\epsilon$ tolerance is at most
\begin{equation*}
	k = \mathcal{O}(\sqrt{\kappa(A)}\log\frac{1}{\epsilon}).
\end{equation*}
The cost of each CG iteration is $\Theta(dN)$ and the upper bound of the condition number for the discretized system is %$\kappa(A) = \Theta(h^{-2}) = \Theta(N^{2/d})$ 
$\kappa(A) = \Theta(\operatorname{poly}(d) h^{-2}) = \Theta(\operatorname{poly}(d) N^{2/d})$  \cite{brenner2008mathematical, jiang2025polynomial}. 
Then, the total complexity of the classical FEM using CG is
\begin{equation*}
	\mathcal{O}(dNk) = \mathcal{O}(dN \sqrt{\kappa(A)}\log\frac{1}{\epsilon} ) = %\mathcal{O}(dN N^{\frac{1}{d}} \log \frac{1}{\epsilon}) =
	 \mathcal{O}(\operatorname{poly}(d)  N^{\frac{d+1}{d}} \log \frac{1}{\epsilon}).
\end{equation*}
Finally, using the relationship $N=\Theta(\epsilon^{-d})$~\eqref{eqn:estimate-N}, the total complexity becomes
\begin{equation*}
	\mathcal{O}(\operatorname{poly}(d) {\epsilon}^{-(d+1)} \log \frac{1}{\epsilon}),
\end{equation*}
which implies the complexity exponentially depends on the dimension, which leads to the so-called \emph{``curse of dimensionality"}.

Similar PDE problems and discretizations yield equivalent results.  Often an ``optimal'' preconditioner can be applied in order improve the complexity.  For example, in an ideal setting, it is possible to construct a \textit{geometric multigrid} (GMG) method as a preconditioner $B$, such that $\kappa(BA) = \Theta(\operatorname{poly}(d))$ \cite{jiang2025polynomial}.   However, since MG is applied implicitly, each preconditioned CG (PCG) iteration still costs $\Theta(dN)$, and one would obtain the overall complexity 
\begin{equation*}
	 	\mathcal{O}(dN k) = \mathcal{O}(\operatorname{poly}(d)  N \log \frac{1}{\epsilon}) = \mathcal{O}( \operatorname{poly}(d)  \epsilon^{-d} \log \frac{1}{\epsilon} ),
 \end{equation*}
 which is still exponential in $d$.

\subsection{Quantum FEM and its Complexity Analysis} 
As the state-of-the-art classical algorithm can not overcome the \emph{``curse of dimensionality"}, it is natural to ask whether a quantum algorithm can.  As discussed above, in the context of the FEM, a quantum algorithm is employed at the linear solver stage.
Thus, similarly to the classical FEM, the complexity of quantum FEM is mainly determined by the complexity of using quantum algorithms to solve the linear system.  Following \cite{montanaro2016quantum}, the complexity of using the quantum linear solver proposed in \cite{childs2017quantum} for solving a linear system to a given accuracy $\epsilon$ is
\begin{equation*}
	\mathcal{O}(\frac{s \kappa(A)}{\epsilon}  \operatorname{polylog}(N \frac{s \kappa(A)}{\epsilon}) ),
\end{equation*}
where $s$ is the maximal number of nonzero entries in each row of $A$. In the linear FEM setting, $s = \Theta(d)$, $\kappa(A) = \Theta(\operatorname{poly}(d)\epsilon^{-2})$, and $N = \Theta(\epsilon^{-d} )$. Therefore, the total complexity of a quantum FEM is
\begin{equation*}
	\mathcal{O}( \frac{d \operatorname{poly}(d) \epsilon^{-2}}{\epsilon} \operatorname{polylog}( \epsilon^{-d}  \frac{d \epsilon^{-2}}{\epsilon}) ) = \mathcal{O} (\operatorname{poly}(d) \epsilon^{-3} \operatorname{polylog}(\frac{1}{\epsilon}) ).
\end{equation*}
Now we see the exponential speed up in terms of the dimension.   More precisely, to achieve accuracy $\epsilon$ in spatial dimension $d$, the runtime of the classical FEM scales as $\mathcal{O}(\epsilon^{-(d+1)})$, while the quantum FEM scales as $\mathcal{O}(\epsilon^{-3})$.  Therefore, for higher-dimensional problems, the quantum speedup would be significant. 

%\textcolor{red}{(I notice that we need to be careful about the dimension $d$ hidden in the constants. I think we should investigate it carefully. --XH)}

\section{Quantum-inspired Classical Algorithms for FEM using RCD} \label{sec:RCD-FEM}
As mentioned in the introduction, we are interested in whether a quantum-inspired classical algorithm can achieve complexity comparable to the fully quantum algorithm described above. To address this question, we consider the randomized coordinate descent (RCD) method \cite{leventhal2010randomized}, which is the current state-of-the-art quantum-inspired algorithm for solving sparse SPD linear systems and achieves comparable complexity in terms of the number of degrees of freedom. RCD is a natural candidate in this setting because it requires only sampling (S) and query (Q) access to the matrix entries, placing it squarely within the SQ access model under which quantum-inspired algorithms are typically analyzed. Moreover, its convergence rate is governed by the scaled condition number of the system $\kappa_F(A):= \| A \|_F^2/ \lambda_{\min}(A)$, so that to reach a target accuracy it requires $\mathcal{O}(\kappa_F(A) \log(1/\epsilon))$ iterations, each of which can be carried out using only a constant number of sampling and query operations. As a result, its overall complexity inherits the conditioning of the underlying finite-element discretization, which makes it well suited for a direct comparison against the dimension-independent complexity of the quantum algorithms.

We briefly recall the RCD method shown in \cref{alg:RCD-A}.

\begin{algorithm}[htbp]
\caption{Randomized Coordinate Descent}
\label{alg:RCD-A}
\begin{algorithmic}[1]
    \STATE Initialize \( k \leftarrow 0 \);
    \WHILE{$k\leq K$}
        \STATE Choose $i_k$ from $\{1,2,3,...,N\}$ with equal probability;
         \STATE \( \bm{u}^{k+1} = \bm{u}^k + 	\omega_{i_k} \frac{\bm{f}(i_k) - A(i_k,:)\bm{u}^k}{A(i_k, i_k)} \bm{e}_{i_k}  \);
         \STATE \( k \leftarrow k+1\);
    \ENDWHILE
    \RETURN \( \bm{u}^{K+1} \)
\end{algorithmic}
\end{algorithm}
\noindent Again, $\bm{u}^k$ denotes the solution of the $k$-th iterate and $\bm{e}_{i_k}$ is the standard unit basis vector with $1$ in the $i_k$-th entry. We choose a weight, $0< \omega_{i_k} <2$, and $i_k$ is chosen randomly using a probability distribution $p_i$ satisfying $\sum_{i}^{N} p_i = 1$. Specific choices of $\omega_{i_k}$ and $p_i$ are given later based on the theoretical analysis.  

% we discuss the convergence results when the matrix $A$ comes from discretizing \eqref{eqn:Poisson}-\eqref{eqn:Poisson-NBC} via the linear FEM.  We establish both upper and lower bounds on the complexity of this quantum-inspired FEM. 
% % Thus, allow us to answer the question negatively. 
We start our theoretical study of quantum-inspired FEM by investigating the use of RCD to solve the discrete linear system \eqref{eqn:Ax=b} resulting from a finite-element discretization of the Poisson equation, \eqref{eqn:Poisson}-\eqref{eqn:Poisson-NBC}. To see if ``preconditioning'' can help, we then consider an expanded system, which represents a complete MG hierarchy of the discretized system and use RCD on that larger problem.  In both cases, we derive the convergence rate of RCD within the FEM framework, which leads to an upper bound on the computational complexity.  This result reveals that there is no exponential speedup in dimension in the worst-case scenario.  Next, we derive a convergence lower bound for RCD, yielding a corresponding complexity lower bound.  This then shows that, to achieve the $\epsilon$-accuracy, the required complexity grows at least exponentially with dimension. Together, these findings demonstrate that quantum algorithms are superior than quantum-inspired classical algorithm in the FEM setting for solving high-dimensional PDEs.

\subsection{Original (Unpreconditioned) System}\label{sec:original}
For the linear system~\eqref{eqn:Ax=b}, the convergence of RCD developed in~\cite{leventhal2010randomized} can be directly employed.  
Recall the following convergence results of RCD.
%\begin{theorem}
%	If $\omega_i$ and $p_i$	 are chosen such that $\sum_{i=1}^N\frac{p_i\omega_{i}(2-\omega_{i})}{A(i,i)} = \frac{1}{\lambda}$, then the RCD applied on the original linear system~\eqref{eqn:Ax=b} has the following error decay in the expectation,
%	\begin{equation} \label{ine:RCD-error-expectation-A}
%		\mathbb{E}(\| \bm{x} - \bm{x}^{k+1} \|_A^2)  \leq \left(  1 -   \frac{\lambda_{\min}(A)}{\lambda} \right)  \mathbb{E}(\| \bm{x} - \bm{x}^k \|_A^2),
%	\end{equation}
%	where $\lambda_{\min}(A)$ is the smallest nonzero eigenvalue of $A$.  
%\end{theorem}
\begin{lemma}[\cite{leventhal2010randomized}]\label{thm:RCD-error-expectation-A}
	If $\displaystyle p_i = \frac{A(i,i)}{\operatorname{tr}(A)}$ and $\omega_{i} = 1$,  then RCD applied on the original linear system~\eqref{eqn:Ax=b} has the following error decay in expectation,
	\begin{equation} \label{ine:RCD-error-expectation-A}
		\mathbb{E}(\| \bm{u} - \bm{u}^{k+1} \|_A^2)  \leq \left(  1 -   \frac{\lambda_{\min}(A)}{\operatorname{tr}(A)} \right) \mathbb{E}( \| \bm{u} - \bm{u}^k \|_A^2),
	\end{equation}
	where $\lambda_{\min}(A)$ is the smallest nonzero eigenvalue of $A$ and $\operatorname{tr}(A)$ is the trace of $A$.
\end{lemma}

\begin{remark} \label{rem:RCD-expend-importance-sampling-A}
	In the FEM setting, based on the analysis in \cite{jiang2025polynomial}, by using the Poincar\'{e} inequality, we know the estimate for the smallest eigenvalue of $A$,
	\begin{align}\label{eigmin_bound}
		\lambda_{\min}(A) = \Theta( \operatorname{poly}(d)h^{d}) =  \Theta( \operatorname{poly}(d) N^{-1}),
	\end{align} 
	and, in addition, since $\operatorname{tr}(A) = \sum_{i}^N A(i,i)$ and $A(i,i)= \Theta(d h^{d-2})$ \cite{brenner2008mathematical}, we have
	\begin{align}\label{trace_bound}
		\operatorname{tr}(A) = \Theta(dNh^{d-2}) = \Theta(dh^{-2}) = \Theta(dN^{2/d}).
	\end{align}
Combining \eqref{eigmin_bound}-\eqref{trace_bound} with \eqref{eqn:estimate-N}, \eqref{ine:RCD-error-expectation-A} becomes
	\begin{align*}
		\mathbb{E}(\| \bm{u} - \bm{u}^{k+1} \|_A^2)  & \leq \left(  1 -   \Theta(\operatorname{poly}(d)N^{-(1+\frac2d)}) \right)  \mathbb{E}(\| \bm{u} - \bm{u}^k \|_A^2) \\
        & = \left(  1 -   \Theta(\operatorname{poly}(d) \epsilon^{d+2}) \right)  \mathbb{E}(\| \bm{u} - \bm{u}^k \|_A^2).
	\end{align*}
	This, as we discuss later, leads to an upper bound of the computational complexity that still depends on the dimension $d$ exponentially.  
	\end{remark}
	
\subsubsection{A Complexity Upper Bound} \label{sec:original-upper-bound}
To derive the upper bound, we assume that the initial guess is $\bm{u}^0 = \bm{0}$.  The following theorem summarizes a computational complexity upper bound.

\begin{theorem}\label{thm-orig-upper}
	For the original linear system \eqref{eqn:Ax=b} arising from the\newline $d$-dimensional Poisson problem \eqref{eqn:Poisson}-\eqref{eqn:Poisson-NBC}, assume that $u$ and $u_h^k$ are the exact solution and finite-element solution via quantum-inspired FEM, satisfying $\|u-u_h^k\|/\|f\| = \mathcal{O}(\operatorname{poly}(d) \epsilon)$. The overall computational complexity is at most\newline $\displaystyle \mathcal{O}(\operatorname{poly}(d) \epsilon^{-(d+2)}\log(\epsilon^{-1}))$ with probability at least $0.99$.
\end{theorem}

\begin{proof}
Since we choose $\bm{u}^0 = \bm{0}$, based on \Cref{thm:RCD-error-expectation-A}, we have
\begin{equation*}
	\mathbb{E}(\| \bm{u} - \bm{u}^k \|^2_A) \leq \left( 1 - \frac{\lambda_{\min}(A)}{\operatorname{tr}(A)}  \right)^k \| \bm{u} \|_A^2.
\end{equation*}
By Markov's inequality, i.e., $\mathbb{P}(X \geq a) \leq \frac{\mathbb{E}(X)}{a}$, there holds
\begin{equation*}
	\mathbb{P}\left(\| \bm{u} - \bm{u}^k \|^2_{A} \geq 100 \mathbb{E}(\| \bm{u} - \bm{u}^k \|^2_{A}) \right)	\leq \frac{1}{100}.
\end{equation*}
This means that, with high probability $0.99$, we obtain
\begin{equation*}
	\| \bm{u} - \bm{u}^k \|^2_{A} \leq 100 \, \mathbb{E}(\| \bm{u} - \bm{u}^k \|^2_{A}) \leq 100 \left( 1 - \frac{\lambda_{\min}(A)}{\operatorname{tr}(A)}  \right)^k \| \bm{u} \|_A^2.
\end{equation*}
Similarly to the convergence analysis of CG in \Cref{subsec:CFEM-complexity}, the above inequality becomes 
\begin{equation*}
	\| \bm{u} - \bm{u}^k \|^2_{A} \leq C \operatorname{poly}(d) \left( 1 - \frac{\lambda_{\min}(A)}{\operatorname{tr}(A)}  \right)^k \| \bm{u} \|_{2}^2 \leq C \operatorname{poly}(d) \left( 1 - \frac{\lambda_{\min}(A)}{\operatorname{tr}(A)}  \right)^k \| f \|^2 ,
\end{equation*}
and, thus, to achieve the $\epsilon$-accuracy,  the number of iterations for RCD should be
\begin{equation*}
	k = \mathcal{O}(\frac{\operatorname{tr}(A)}{\lambda_{\min}(A)} \log \frac{1}{\epsilon}).
\end{equation*}
Inserting the values of the smallest eigenvalue \eqref{eigmin_bound} and the trace of the matrix \eqref{trace_bound}, we obtain a bound of
\begin{equation*}
	k = \mathcal{O}( \operatorname{poly}(d) N^{1+\frac{2}{d}} \log \frac{1}{\epsilon}).
\end{equation*}
Finally, since each step of RCD costs $\Theta(d)$, the overall complexity of the quantum-inspired FEM is, with high probability $0.99$, 
\begin{equation}
	\mathcal{O}(dk) = \mathcal{O}(d \operatorname{poly}(d) N^{1+\frac{2}{d}} \log \frac{1}{\epsilon}) = \mathcal{O}( \operatorname{poly}(d)\epsilon^{-(d+2)} \log \frac{1}{\epsilon}).
\end{equation}
This completes the proof.
\end{proof}

The result from \cref{thm-orig-upper} implies that the complexity of the FEM with RCD (one type of quantum-inspired FEM) still depends on the dimension $d$ exponentially and still suffers from the ``\textit{curse of dimensionality}''.

\subsubsection{A Complexity Lower Bound}\label{sec:original-lower-bound}
The upper bound provided in the previous section implies that, in the worst-case scenario, the complexity depends exponentially on $d$. In this section, we show that to achieve $\epsilon$-accuracy, the required computational cost must also depend exponentially on $d$. To this end, we provide the following convergence lower bound of RCD. 
%following the idea proposed in \cite[Theorem 3.1]{2025Provable}.
\begin{lemma}
\label{thm-lower-bound-origin}
	When $\displaystyle p_i = \frac{A(i,i)}{\operatorname{tr}(A)}$ and $\omega_{i} = 1$, the RCD method applied on the original linear system \eqref{eqn:Ax=b} has the following lower bound estimate in expectation,
	\begin{align} \label{ine:RCD-error-expectation-A-lower}
		\mathbb{E}(\| \bm{u} - \bm{u}^{k+1} \|_A^2)  \geq  \left(  1 - \frac{\lambda_{\max}(A)}{\operatorname{tr}(A)} \right)  \mathbb{E}( \| \bm{u} - \bm{u}^k \|_A^2).
	\end{align}
\end{lemma}
\begin{proof}
	For the RCD method, by direct calculation,
%	we rewrite the step 4 of \cref{eqn:RCD-A} as follows,
%	\begin{align}\label{eq-numer}
%		\bm{u}^{k+1} =  \bm{u}^k + \omega_{i_k} \frac{\bm{e}_{i_k}\bm{e}_{i_k}^T(\bm{b} - A\bm{x}^k)}{A(i_k, i_k)} . 
%	\end{align}
%	Assume $\bm{x}$ is the exact solution satisfying
%	\begin{align}\label{eq-exact}
%		\bm{x} = \bm{x} + \omega_{i_k} \frac{\bm{e}_{i_k}\bm{e}_{i_k}^T(\bm{b} - A\bm{x})}{A(i_k, i_k)}.
%	\end{align}
%	Subtract \eqref{eq-exact} from \eqref{eq-numer}, 
	\begin{align*}
		\| \bm{u} - \bm{u}^{k+1} \|_A^2 
		%& =  (\left(I- \omega_{i_k} \frac{\bm{e}_{i_k}\bm{e}_{i_k}^T A}{A(i_k, i_k)}\right)(\bm{u} - \bm{u}^{k}), \left(I- \omega_{i_k} \frac{\bm{e}_{i_k}\bm{e}_{i_k}^T A}{A(i_k, i_k)}\right)(\bm{u} - \bm{u}^{k}) )_A \\
 	%& = (\left(I- \omega_{i_k} \frac{\bm{e}_{i_k}\bm{e}_{i_k}^T A}{A(i_k, i_k)}\right)^{T}\left(I- \omega_{i_k} \frac{\bm{e}_{i_k}\bm{e}_{i_k}^T A}{A(i_k, i_k)}\right)(\bm{u} - \bm{u}^{k}), \bm{u} - \bm{u}^{k} )_A \\
		& = ( \left(  I - \frac{\omega_{i_k} (2- \omega_{i_k})}{A(i_k, i_k)} \bm{e}_{i_k} \bm{e}_{i_k}^T A \right)(\bm{u} - \bm{u}^k), \bm{u} - \bm{u}^k )_A.
	\end{align*}
	Taking the conditional expectation of the error,
	\begin{align}\label{eq-expect}
		\mathbb{E}(\| \bm{u} - \bm{u}^{k+1} \|_A^2 \vert \bm{u}^k ) &= \sum_{i=1}^N p_i ( \left(  I - \frac{\omega_{i} (2- \omega_{i})}{A(i, i)} \bm{e}_{i} \bm{e}_{i}^T A \right)(\bm{u} - \bm{u}^k), \bm{u} - \bm{u}^k )_A \\
		& = (  \left( I - \sum_{i=1}^N \frac{p_i\omega_{i} (2- \omega_{i})}{A(i, i)} \bm{e}_{i} \bm{e}_{i}^T A  \right) (\bm{u} - \bm{u}^k), \bm{u} - \bm{u}^k )_A. \nonumber
	\end{align}
	Using the fact that $\displaystyle p_i = \frac{A(i,i)}{\operatorname{tr}(A)}$ and $\omega_{i} = 1$, it holds that 
%	\begin{align*}
%		I - \sum_{i=1}^N \frac{p_i\omega_{i} (2- \omega_{i})}{A(i, i)} \bm{e}_{i} \bm{e}_{i}^T A & = I - \frac{A}{\operatorname{tr}(A)} = I - \frac{2}{\operatorname{tr}(A)}A + \frac{1}{\operatorname{tr}(A)}A \\
%		%& = I - \frac{2}{\operatorname{tr}(A)}A + \Big(\frac{1}{\operatorname{tr}(A)}A\Big)^2 + \frac{1}{\operatorname{tr}(A)}A - \Big(\frac{1}{\operatorname{tr}(A)}A\Big)^2 \\
%		& = \Big( I - \frac{1}{\operatorname{tr}(A)}A\Big)^2 + \left( \frac{1}{\operatorname{tr}(A)}A - \Big(\frac{1}{\operatorname{tr}(A)}A\Big)^2 \right) \\
%		& \succeq \Big( I - \frac{1}{\operatorname{tr}(A)}A\Big)^2.
%	\end{align*}
%	Here the notation $A \succeq B$ means that $A-B$ is SPD and the last inequality holds since the matrix $\displaystyle \frac{1}{\operatorname{tr}(A)}A - \Big(\frac{1}{\operatorname{tr}(A)}A\Big)^2$ is SPD, which is due to the fact that 
%	\begin{align*}
%		(\left(\frac{1}{\operatorname{tr}(A)}A - \Big(\frac{1}{\operatorname{tr}(A)}A\Big)^2\right)\omega, \omega)_A & = \frac{1}{\operatorname{tr}(A)}(\left(I - \frac{A}{\operatorname{tr}(A)}\right)A\omega, A\omega) \\
%		& \geq \frac{1}{\operatorname{tr}(A)}\left(1 - \frac{\lambda_{\max}(A)}{\operatorname{tr}(A)}\right)\|A\omega\|^2 > 0, \quad \forall\,\omega\in\mathbb{R}^n. 
%	\end{align*}
%	Combing \eqref{eq-expect} with the above equations yields
	\begin{align*}
		\mathbb{E}(\| \bm{u} - \bm{u}^{k+1} \|_A^2 \vert \bm{u}^k ) & = (\Big( I - \frac{1}{\operatorname{tr}(A)}A\Big)(\bm{u}-\bm{u}^k), \bm{u} - \bm{u}^k)_A \\
	& \geq \Big( 1 - \frac{\lambda_{\max}(A)}{\operatorname{tr}(A)}\Big)\|\bm{u}-\bm{u}^k)\|_A^2,
	\end{align*}
	which leads to the lower bound in expectation \eqref{ine:RCD-error-expectation-A-lower}.
\end{proof}

Based on the convergence lower bound of RCD in \Cref{thm-lower-bound-origin}, we further obtain the corresponding lower bound of computational complexity, stated in the following theorem.

\begin{theorem}\label{thm-orig-lower}
	Under the assumption of \Cref{thm-orig-upper}, the lower bound of overall computational complexity for the RCD method applied above is $\displaystyle \Omega (\operatorname{poly}(d)  \epsilon^{-d}\log(\epsilon^{-1}))$ in expectation.
\end{theorem}
\begin{proof}
 The lower bound result \eqref{ine:RCD-error-expectation-A-lower} implies that, if $\bm{u}^0 = \bm{0}$, we have
\begin{align*}
	\mathbb{E}(\| \bm{u} - \bm{u}^{k} \|_A^2)  \geq  \left(  1 - \frac{\lambda_{\max}(A)}{\operatorname{tr}(A)} \right)^{k} \| \bm{u}  \|_A^2. 
\end{align*}
Utilizing the inequality $\displaystyle (1-x)^k \geq \operatorname{exp}(-\frac{kx}{1-x})$, $0<x<1$, gives the complexity 
% \textcolor{red}{
% \begin{align*}
%     (1-\frac{\lambda_{\min}(A)}{\operatorname{tr}(A)})^k & \geq \exp\left(-\frac{k\frac{\lambda_{\min}(A)}{\operatorname{tr}(A)}}{1-\frac{\lambda_{\min}(A)}{\operatorname{tr}(A)}}\right) \\
%     & = \exp\left(-\frac{k \lambda_{\min}(A)}{\operatorname{tr}(A) - \lambda_{\min}(A)}\right)  \geq \epsilon,
% \end{align*}
% then, taking $\log$ on both sides leads to
% \begin{align*}
%     & -\frac{k \lambda_{\min}(A)}{\operatorname{tr}(A) - \lambda_{\min}(A)} \geq \log\epsilon \Rightarrow 
%     \frac{k \lambda_{\min}(A)}{\operatorname{tr}(A) - \lambda_{\min}(A)} \leq \log\frac{1}{\epsilon} \\
%     & \Rightarrow k \leq \frac{\operatorname{tr}(A) - \lambda_{\min}(A)}{\lambda_{\min}(A)} \log \frac{1}{\epsilon}.
% \end{align*}
% }
% after 
\begin{align*}
k = \Omega(\frac{\operatorname{tr}(A) - \lambda_{\max}(A)}{\lambda_{\max}(A)} \log \frac{1}{\epsilon}).
\end{align*}
In other words, even after $k$ iterations, the relative error in expectation is still greater than $\epsilon$.  Since each step of RCD costs $\Theta(d)$,  $\lambda_{\max}(A) = \Theta(\operatorname{poly}(d) h^{d-2})$, and $\operatorname{tr}(A) = \Theta(d h^{-2})$,  the lower bound of complexity is 
\begin{align}\label{eqn:lower-bound}
\Omega(\operatorname{poly}(d) (\epsilon^{-d}-1) \log \frac{1}{\epsilon}),
\end{align}
which completes the proof.
%still depends on $d$ exponentially. 
\end{proof}

In summary, \cref{thm-orig-lower} shows that, even with computational complexity that is exponential in $d$, the relative error in expectation is still greater than $\epsilon$, which implies that the ``\textit{curse of dimensionality}'' is unavoidable.

\subsection{Expanded (Preconditioned) System}\label{sec:expaned}
%In order to take advantage of existing quantum-inspired algorithms, we first focus on reducing the condition number of $A$.  A well-known strategy is using a preconditioner $B$ such that $\kappa(BA) \ll \kappa(A)$ and, idealy, $\kappa(AB) = \mathcal{O}(1)$.  
As discussed in \Cref{sec:preliminaries}, introducing a preconditioner can improve the overall performance of both the classical FEM and the quantum FEM. In this section, we therefore explore the use of a preconditioner in combination with RCD.  

Since our focus is on the Poisson equation, a natural choice is the GMG method, which is a well-known optimal preconditioner in this case.  However, since MG defines the preconditioner $B$ implicitly, we cannot just explicitly form $B$, which is dense, and use it to construct a preconditioned linear system $BA \bm{u} = B \bm{f}$ directly, especially for the purpose of applying RCD.  To address this issue, we adopt an equivalent representation of the MG method to achieve our goal, i.e., the so-called \emph{expanded linear system} \cite{griebel1994multilevel}.   Let $P_{\ell} \in \mathbb{R}^{N \times N_{\ell}}$ be the composite prolongation operator from the $\ell$-th coarse level to level $1$ (finest level) used in a $L$-level MG method, where $N = N_1$.  We then define the operator $\Pi$ as follows
\begin{equation*}
	\Pi = [I, P_2, P_3, \cdots, P_L]. 
\end{equation*}
Then, the expanded system is defined as
\begin{equation} \label{eqn:expanded-system} 
	\widetilde{A} \widetilde{\bm{u}} = \widetilde{\bm{f}},
\end{equation}
where $\widetilde{A} = \Pi^T A \Pi \in \mathbb{R}^{\widetilde{N} \times \widetilde{N}}$ 
% $$
% \begin{pmatrix}
	%  A_1 & A_1 P_1 & \cdots A_1 P_L \\
	%  P^T_1 A_1 & A_2 & P^T_1 A_1 P_1 & \cdots & \\
	% \end{pmatrix} \\
% $$
and $\widetilde{\bm{f}} = \Pi^T \bm{f} \in \mathbb{R}^{\widetilde{N}}$ with $\widetilde{N} = \sum_{\ell=1}^L N_{\ell}$. It is straightforward to verify that we can recover $\bm{u}$ from $\widetilde{\bm{u}}$ via $\bm{u} = \Pi \widetilde{\bm{u}}$.   Therefore, instead of solving~\eqref{eqn:Ax=b}, we can equivalently solve~\eqref{eqn:expanded-system}.  
%In fact, applying one step of a MG V-cycle on~\eqref{eqn:Ax=b} is equivalent to one step of Gauss-Seidel on~\eqref{eqn:expanded-system} under certain conditions.  
The expanded system~\eqref{eqn:expanded-system} gives an alternative representation of the MG preconditioned system $BA \bm{u} = B\bm{f}$, while maintaining the ``preconditioning" effect.  Additionally, although $\widetilde{A}$ is symmetric semi-definite, we can fully characterize its null space.  Its effective condition number with Jacobi preconditioner (excluding the null space) satisfies $\kappa(\widetilde{D}^{-1}\widetilde{A}) = \kappa(BA)$ where $\widetilde{D} = \operatorname{diag}(\widetilde{A})$ and $B$ is a variant of the BPX preconditioner \cite{bramble1990parallel}.  Thus, for our Laplace example, $\kappa(\widetilde{D}^{-1}\widetilde{A}) = \Theta(\operatorname{poly}(d))$ \cite{griebel2013multilevel, bramble1990parallel, jiang2025polynomial}.

%Therefore, we propose to use the expanded system for a general class of second-order PDEs.  In the finite element setting, we can build~\eqref{eqn:expanded-system} directly and then apply some existing quantum-inspired algorithm.  Since $\widetilde{A}$ is in general sparse since it is obtained by discretize PDEs, we plan to focus on the second type quantum-inspired algorithms that is based on randomized iterative methods.  For the Laplace example we discuss here,

Note that when we apply RCD to the expanded system \eqref{eqn:expanded-system}, the iteration step (Line 4) in \cref{alg:RCD-A} becomes
	\begin{equation} \label{eqn:RCD}
		\widetilde{\bm{u}}^{k+1} = \widetilde{\bm{u}}^k + \omega_{i_k} \, \frac{\widetilde{\bm{f}}(i_k) - \widetilde{A}(i_k, :) \widetilde{\bm{u}}^k }{\widetilde{A}(i_k, i_k)} \bm{e}_{i_k}. %\quad \text{where the row $r_k$ is chosen with probability $\frac{\widetilde{A}(r_k, r_k)}{\operatorname{trace}(\widetilde{A})}$}.
	\end{equation}
	%Here $\widetilde{\bm{u}}^k$ denotes the $k$-th iterates and $\bm{e}_{i_k}$ is all zero vector except that the $i_k$-th entry is $1$.    $0< \omega^k <2$ is the a proper chosen weights and the row $i_k$ is chosen using the probability distribution $p_i$, $\sum_{i}^{N} p_i = 1$. We will specify our choices of $\omega^k$ and $p_i$ later based on the theoretical analysis.
	Since the expanded system \eqref{eqn:expanded-system} is positive and \emph{semi-definite}, the theoretical analysis of RCD is slightly different, though the final result is quite similar.  In addition, to relate RCD with the Jacobi-preconditioned expanded system $\widetilde{D}^{-1} \widetilde{A}$, motivated by \cite{hu2019randomized}, we simply choose the probability $p_i = 1/\widetilde{N}$.  We present the convergence analysis of applying RCD to the expanded system~\eqref{eqn:expanded-system} in the following theorem. 
	
	\begin{theorem} \label{thm:RCD-expanded-cnvergence}
		%If $\omega_i$ and $p_i$	 are chosen such that $\frac{p_i\omega_{i}(2-\omega_{i})}{\widetilde{A}(i,i)} = \frac{1}{\lambda}$, 
		If $\displaystyle p_i = 1/\widetilde{N}$
		% \frac{\widetilde{A}(i,i)}{\operatorname{tr}(\widetilde{A})}$ 
		and $\omega_{i} = 1$, 
		then the RCD applied on the expanded system~\eqref{eqn:expanded-system} has the following error decay in expectation for the original linear system~\eqref{eqn:Ax=b},
		\begin{equation} \label{ine:RCD-error-expectation}
			\mathbb{E}(\| \bm{u} - \bm{u}^{k+1} \|_A^2)  \leq \left(  1 -   \frac{\lambda_{\min}(\widetilde{D}^{-1}\widetilde{A})}{\widetilde{N}} \right) \mathbb{E}( \| \bm{u} - \bm{u}^k \|_A^2),
		\end{equation}
		where $\lambda_{\min}(M)$ denotes the smallest nonzero eigenvalue of a matrix $M$ if $M$ is positive semi-definite.  
	\end{theorem}
    %\textcolor{blue}{JA: is that statement fully correct?}
	\begin{proof}
		%We first rewrite~\eqref{eqn:RCD} as follows,
		%\begin{equation*}
		%	\widetilde{\bm{x}}^{k+1} = \widetilde{\bm{x}}^k + \omega_{i_k} \frac{1}{\widetilde{A}(i_k, i_k)} \bm{e}_{i_k}\bm{e}_{i_k}^T \left( \widetilde{\bm{b}} - \widetilde{A} \widetilde{\bm{x}}^k  \right).
		%\end{equation*}
		%Note that the solution $\widetilde{\bm{x}}$ satisfies,
		%\begin{equation*}
		%	\widetilde{\bm{x}} = \widetilde{\bm{x}} + \omega_{i_k} \frac{1}{\widetilde{A}(i_k, i_k)} \bm{e}_{i_k}\bm{e}_{i_k}^T \left( \widetilde{\bm{b}} - \widetilde{A} \widetilde{\bm{x}} \right).
		%\end{equation*}
		By direct calculation, we have the following identity for the error $\widetilde{\bm{u}} - \widetilde{\bm{u}}^{k+1} $,
		\begin{equation*}
			\widetilde{\bm{u}} - \widetilde{\bm{u}}^{k+1} = \left(  I - \frac{\omega_{i_k}}{\widetilde{A}(i_k,i_k)}  \bm{e}_{i_k}\bm{e}_{i_k}^T \widetilde{A}  \right) ( \widetilde{\bm{u}} - \widetilde{\bm{u}}^{k}   ).
		\end{equation*}
		Using the relation between expanded system and the original system gives,
		\begin{align}\label{u-tilde-u}
			\| \bm{u} - \bm{u}^k \|_A^2 & = (A (\bm{u} - \bm{u}^k), (\bm{u} - \bm{u}^k)) = (A \Pi(\widetilde{\bm{u}} - \widetilde{\bm{u}}^k), \Pi (\widetilde{\bm{u}} - \widetilde{\bm{u}}^k) ) \\
            & = (\widetilde{A} (\widetilde{\bm{u}} - \widetilde{\bm{u}}^k), (\widetilde{\bm{u}} - \widetilde{\bm{u}}^k)). \nonumber
		\end{align}
		Then, combining with the above two equations, we obtain
		\begin{align*}
			& \quad \| \bm{u} - \bm{u}^{k+1} \|_A^2 
			%& = (\widetilde{A} (\widetilde{\bm{x}} - \widetilde{\bm{x}}^{k+1}), (\widetilde{\bm{x}} - \widetilde{\bm{x}}^{k+1}))  
			\\
			& =  (\widetilde{A}  \left(  I - \frac{\omega_{i_k}}{\widetilde{A}(i_k,i_k)}  \bm{e}_{i_k}\bm{e}_{i_k}^T \widetilde{A}  \right) ( \widetilde{\bm{u}} - \widetilde{\bm{u}}^{k}   ), \left(  I - \frac{\omega_{i_k}}{\widetilde{A}(i_k,i_k)}  \bm{e}_{i_k}\bm{e}_{i_k}^T \widetilde{A}  \right) ( \widetilde{\bm{u}} - \widetilde{\bm{u}}^{k}   )   ) \\
			& = ( \left(  I - \frac{\omega_{i_k}}{\widetilde{A}(i_k,i_k)}  \bm{e}_{i_k}\bm{e}_{i_k}^T \widetilde{A}  \right)^2 (\widetilde{\bm{u}} - \widetilde{\bm{u}}^{k}), \widetilde{A} (\widetilde{\bm{u}} - \widetilde{\bm{u}}^{k}) ) \\
			& = ( \left(  I - \frac{\omega_{i_k}(2-\omega_{i_k})}{\widetilde{A}(i_k,i_k)}  \bm{e}_{i_k}\bm{e}_{i_k}^T \widetilde{A}  \right) (\widetilde{\bm{u}} - \widetilde{\bm{u}}^{k}), \widetilde{A} (\widetilde{\bm{u}} - \widetilde{\bm{u}}^{k}) ).
		\end{align*}
		%Note that 
		%\begin{equation*}
		%	\left(  I - \frac{\omega_{i_k}}{\widetilde{A}(i_k,i_k)}  \bm{e}_{i_k}\bm{e}_{i_k}^T \widetilde{A}  \right)^2 = I - \frac{\omega_{i_k}(2-\omega_{i_k})}{\widetilde{A}(i_k,i_k)}\bm{e}_{i_k}\bm{e}_{i_k}^T \widetilde{A},
		%\end{equation*}
		%This leads to 
	%	\begin{align*}
		%	\| \bm{u} - \bm{u}^{k+1} \|_A^2 & = ( \left(  I - \frac{\omega_{i_k}(2-\omega_{i_k})}{\widetilde{A}(i_k,i_k)}  \bm{e}_{i_k}\bm{e}_{i_k}^T \widetilde{A}  \right) (\widetilde{\bm{u}} - \widetilde{\bm{u}}^{k}), \widetilde{A} (\widetilde{\bm{x}} - \widetilde{\bm{x}}^{k}) ).
		%\end{align*}
		Thus, the conditional expectation of the error becomes,
		\begin{align*}
			\mathbb{E}(\| \bm{u} - \bm{u}^{k+1} \|_A^2 \vert \bm{u}^k) & = \sum_{i=1}^{\widetilde{N}} p_i ( \left(  I - \frac{\omega_{i}(2-\omega_{i})}{\widetilde{A}(i,i)}  \bm{e}_{i}\bm{e}_{i}^T \widetilde{A}  \right) (\widetilde{\bm{u}} - \widetilde{\bm{u}}^{k}), \widetilde{A} (\widetilde{\bm{u}} - \widetilde{\bm{u}}^{k}) ) \\
			& = (  \left(  I - \sum_{i=1}^{\widetilde{N}} \frac{p_i\omega_{i}(2-\omega_{i})}{\widetilde{A}(i,i)}  \bm{e}_{i}\bm{e}_{i}^T \widetilde{A}    \right) (\widetilde{\bm{u}} - \widetilde{\bm{u}}^{k}), \widetilde{A} (\widetilde{\bm{u}} - \widetilde{\bm{u}}^{k} )).
		\end{align*}
		%Note that for the RCD  $\displaystyle \frac{p_i\omega_{i}(2-\omega_{i})}{\widetilde{A}(i,i)} = \frac{1}{\operatorname{tr}(\widetilde{A})}$, then
		Due to our choice of $p_i$ and $\omega_i$, the above calculation provides the following estimate
%		\begin{align*}
%			\mathbb{E}(\| \bm{u} - \bm{u}^{k+1} \|_A^2 \vert \bm{u}^k) & = (  \left(  I - \frac{1}{\operatorname{tr}(\widetilde{A})} \sum_{i=1}^N  \bm{e}_{i}\bm{e}_{i}^T \widetilde{A}    \right) (\widetilde{\bm{u}} - \widetilde{\bm{u}}^{k}), \widetilde{A} (\widetilde{\bm{u}} - \widetilde{\bm{u}}^{k} ))  \\
%			& = ((\widetilde{\bm{u}} - \widetilde{\bm{u}}^{k}), \widetilde{A} (\widetilde{\bm{u}} - \widetilde{\bm{u}}^{k} )) - \frac{1}{\operatorname{tr}(\widetilde{A})} (  \widetilde{A} (\widetilde{\bm{u}} - \widetilde{\bm{u}}^{k}), \widetilde{A} (\widetilde{\bm{u}} - \widetilde{\bm{u}}^{k} )) \\
%			& = (\bm{u} - \bm{u}^k, \bm{u} - \bm{u}^k)_A -  \frac{1}{\operatorname{tr}(\widetilde{A})}(\Pi \Pi^T A (\bm{u} - \bm{u}^k), \bm{u} - \bm{u}^k)_A  \\
%			& \leq \left(  1 -   \frac{\lambda_{\min}(\Pi \Pi^T A)}{\operatorname{tr}(\widetilde{A})} \right) \| \bm{u} - \bm{u}^k \|_A^2.
%		\end{align*}
		\begin{align*}
				\mathbb{E}(\| \bm{u} - \bm{u}^{k+1} \|_A^2 \vert \bm{u}^k) & = (  \left(  I - \frac{1}{\widetilde{N}} \sum_{i=1}^{\widetilde{N}} \frac{1}{\widetilde{A}(i,i)}  \bm{e}_{i}\bm{e}_{i}^T \widetilde{A}    \right) (\widetilde{\bm{u}} - \widetilde{\bm{u}}^{k}), \widetilde{A} (\widetilde{\bm{u}} - \widetilde{\bm{u}}^{k} ))  \\
				& = ((\widetilde{\bm{u}} - \widetilde{\bm{u}}^{k}), \widetilde{A} (\widetilde{\bm{u}} - \widetilde{\bm{u}}^{k} )) - \frac{1}{\widetilde{N}} (  \widetilde{D}^{-1} \widetilde{A} (\widetilde{\bm{u}} - \widetilde{\bm{u}}^{k}), \widetilde{A} (\widetilde{\bm{u}} - \widetilde{\bm{u}}^{k} )) \\
				& = (\bm{u} - \bm{u}^k, \bm{u} - \bm{u}^k)_A -  \frac{1}{\widetilde{N}}(\Pi \widetilde{D}^{-1} \Pi^T A (\bm{u} - \bm{u}^k), \bm{u} - \bm{u}^k)_A  \\
				& \leq \left(  1 -   \frac{\lambda_{\min}(\Pi \widetilde{D}^{-1} \Pi^T A)}{\operatorname{tr}(\widetilde{A})} \right) \| \bm{u} - \bm{u}^k \|_A^2.
			\end{align*}
	Noting that $\lambda_{\min}(\Pi\ \widetilde{D}^{-1}\Pi^T A) = \lambda_{\min}( \widetilde{D}^{-1} \Pi^T A \Pi) = \lambda_{\min}(\widetilde{D}^{-1} \widetilde{A})$, there holds %where  $\lambda_{\min}(\widetilde{A})$ denotes the smallest nonzero eigenvalue of $\widetilde{A}$, 
		\begin{equation*}
			\mathbb{E}(\| \bm{u} - \bm{u}^{k+1} \|_A^2 \vert \bm{u}^k)  \leq \left(  1 -   \frac{\lambda_{\min}(\widetilde{D}^{-1}\widetilde{A})}{\widetilde{N}} \right) \| \bm{u} - \bm{u}^k \|_A^2.
		\end{equation*}
		Taking another expectation, we obtain the estimate \eqref{ine:RCD-error-expectation}. 
		%\begin{equation*}
		%	\mathbb{E}(\| \bm{x} - \bm{x}^{k+1} \|_A^2)  \leq \left(  1 -   \frac{\lambda_{\min}(\widetilde{A})}{\lambda} \right) \mathbb{E}( \| \bm{x} - \bm{x}^k \|_A^2).
	%	\end{equation*}
		%Furthermore, if $\lambda = \lambda_{\max}(\widetilde{A})$, then we have
		%\begin{equation*}
		%	\mathbb{E}(\| \bm{x} - \bm{x}^{k+1} \|_A^2)  \leq \left(  1 -   \frac{\lambda_{\min}(\widetilde{A})}{\lambda_{\max}(\widetilde{A})} \right) \mathbb{E}( \| \bm{x} - \bm{x}^k \|_A^2) = \left( 1 - \frac{1}{\kappa(\widetilde{A})} \right) \mathbb{E}( \| \bm{x} - \bm{x}^k \|_A^2).
		%\end{equation*}
		%This completes the proof. 
	\end{proof}
	
	\begin{remark} \label{rem:RCD-expend-importance-sampling}
	%	Usually, we use $p_i = \frac{\widetilde{A}(i,i)}{\operatorname{tr}(\widetilde{A})}$ and $\omega_{i} = 1$ in RCD, see \cite{leventhal2010randomized}.  In this case, the convergence estimate becomes 
	%	\begin{equation*}
	%		\mathbb{E}(\| \bm{x} - \bm{x}^{k+1} \|_A^2)  \leq \left(  1 -   \frac{\lambda_{\min}(\widetilde{A})}{\operatorname{tr}(\widetilde{A})} \right) \mathbb{E}( \| \bm{x} - \bm{x}^k \|_A^2).
%		\end{equation*}
		In the GMG setting, based on \cite{griebel2013multilevel, jiang2025polynomial}, choosing $h = h_1 =\Theta(\epsilon)$ as before,  for the expanded matrix $\widetilde{A}$, we have  
%		\begin{flalign}\label{eigmin_bound-expand}
%      \lambda_{\min}(\widetilde{A})  = \Theta(dh^{d-2}) = \Theta(d \epsilon^{d-2}), \\
%			\operatorname{tr}(\widetilde{A}) = \sum_\ell \operatorname{tr}(A_\ell) = \Theta(\sum_{\ell} 2d N_{\ell} h_{\ell}^{d-2}) = \Theta(d \sum_{\ell} h_{\ell}^{-2}) = \Theta(d h^{-2}) =  \Theta(d \epsilon^{-2}).
%			%\\
%			%& = \mathcal{O}(dN h^{d-2}) = \mathcal{O}(dh^{-d} h^{d-2}) = \mathcal{O}(dh^{-2}) = \mathcal{O}(dN^{2/d}).
%		\end{flalign}
\begin{align}\label{eigmin_bound-expand}
	&\lambda_{\min}( \widetilde{D}^{-1}\widetilde{A})  = \Theta(\operatorname{poly}(d)), \\
&	\widetilde{N} =  \sum_{\ell=1}^L N_{\ell} = \Theta(N) = \Theta(\epsilon^{-d}).
\end{align}
		Then, substituting the above two equations into \eqref{ine:RCD-error-expectation} leads to
		\begin{equation*}
			\mathbb{E}(\| \bm{u} - \bm{u}^{k+1} \|_A^2)  \leq \left(  1 -   \Theta(\operatorname{poly}(d)\epsilon^{d}) \right) \mathbb{E}( \| \bm{u} - \bm{u}^k \|_A^2).
		\end{equation*}
	   Comparing with \Cref{rem:RCD-expend-importance-sampling-A}, we see the effect of the MG preconditioning, which allows us to reduce from $\epsilon^{d+2}$ to $\epsilon^d$.  However, as we discuss in the next section, such a convergence result leads to an upper bound of the computational complexity that still depends on the dimension $d$ exponentially. 
	\end{remark}

\subsubsection{A Complexity Upper Bound}
%Next we consider the RCD for the expanded system~\eqref{eqn:expanded-system}.  From Remark~\ref{rem:RCD-expend-importance-sampling}, using initial guess $\widetilde{\bm{u}}^0 = \bm{0}$, i.e. $\bm{u}^0 = \Pi \widetilde{\bm{u}}^0 = \bm{0}$,
%\begin{equation*}
%	\mathbb{E}(\| \bm{u} - \bm{u}^{k+1} \|_A^2)  \leq \left(  1 -   \frac{\lambda_{\min}(\widetilde{A})}{\operatorname{tr}(\widetilde{A})} \right)^k \| \bm{u} \|^2_A.
%\end{equation*}
Following a similar argument as the complexity upper bound for the original (unpreconditioned) system  (see \Cref{thm:RCD-expanded-cnvergence} in \Cref{sec:original-upper-bound}), for the expanded system, we have the following theorem about its complexity upper bound.

\begin{theorem}\label{thm-expanded-upper}
	For the preconditioned system \eqref{eqn:expanded-system}, arising from the\newline $d$-dimensional Poisson problem \eqref{eqn:Poisson}-\eqref{eqn:Poisson-NBC}, assume that $u$ and $u_h^k$ are the exact solution and finite-element solution via a quantum-inspired FEM, satisfying\newline $\|u-u_h^k\|/\|f\| = \mathcal{O}(\operatorname{poly}(d)\epsilon)$. The overall computational complexity is at most\newline $\displaystyle \mathcal{O}(\operatorname{poly}(d)\epsilon^{-2d}\log^2(\epsilon^{-1}))$ with high probability $0.99$.
\end{theorem}

\begin{proof}
Again, using initial guess $\widetilde{\bm{u}}^0 = \bm{0}$ (which implies $\bm{u}^0 = \bm{0}$) and Markov's inequality, with high probability $0.99$, we have
\begin{equation*}
	\| \bm{u} - \bm{u}^k \|^2_A \leq 
	%C \operatorname{poly}(d) \left(  1 -   \frac{\lambda_{\min}(\widetilde{D}^{-1}\widetilde{A})}{\widetilde{N}} \right)^k  \| \bm{u} \|^2_{2} \leq 
	C \operatorname{poly}(d) \left(  1 -   \frac{\lambda_{\min}(\widetilde{D}^{-1}\widetilde{A})}{\widetilde{N}} \right)^k  \| \bm{f} \|^2.
\end{equation*}
Hence, to achieve desired accuracy, we require the number of iterations to be at most
\begin{equation*}
	k = %\mathcal{O}(\frac{\operatorname{tr}(\widetilde{A})}{\lambda_{\min}(\widetilde{A})} \log \frac{1}{\epsilon}).
	\mathcal{O}(\frac{\widetilde{N}}{\lambda_{\min}(\widetilde{D}^{-1}\widetilde{A})} \log \frac{1}{\epsilon}).
\end{equation*}

Note, we need to be careful about the cost of each iteration of RCD for the expanded system because the sparsity of $\widetilde{A}$ is quite different from the sparsity of $A$. In fact, there are some rows of $\widetilde{A}$ that are dense, which implies, for a general iterator $\widetilde{\bm{u}}^k$, the cost of each step could be $\mathcal{O}(N)$. If we exactly access a dense row, the cost becomes $\Theta(N)$.
However, due to the choice $\widetilde{\bm{u}}^0 = \bm{0}$, at $i$-th iteration of RCD for the expanded system~\eqref{eqn:expanded-system}, $\widetilde{\bm{u}}^i$ is at most $i$-sparse, i.e., at most $i$ components of $\widetilde{\bm{u}}^i$ are nonzeros.  Therefore, the cost of the $i$-th step is at most $\mathcal{O}(i)$ and the overall cost of $k$ steps is
\begin{equation*}
	\mathcal{O}(1 + 2 + \cdots + k) = \mathcal{O}(k^2). 
\end{equation*}
Inserting the value of $k$ gives the overall complexity
\begin{equation*}
%	\mathcal{O}(k^2) = \mathcal{O}(  \frac{\operatorname{tr}^2(\widetilde{A})}{\lambda^2_{\min}(\widetilde{A})} \log^2 \frac{1}{\epsilon^2} ) = \mathcal{O}( \epsilon^{-2d} \log^2 \frac{1}{\epsilon} ),
	\mathcal{O}(k^2) = \mathcal{O}(  \frac{\widetilde{N}^2}{\lambda^2_{\min}(\widetilde{D}^{-1}\widetilde{A})} \log^2 \frac{1}{\epsilon^2} ) = \mathcal{O}( \operatorname{poly}(d) \epsilon^{-2d} \log^2 \frac{1}{\epsilon} ),
\end{equation*}
which follows from the fact that $\lambda_{\min}(\widetilde{D}^{-1}\widetilde{A}) = \Theta(\operatorname{poly}(d))$ and  $\widetilde{N} = \Theta(\epsilon^{-d})$.   This completes the proof.
\end{proof}

Again, \cref{thm-expanded-upper} shows that the complexity for the expanded system still depends on the dimension $d$ exponentially. 
%The conclusion is summarized in the below theorem. 
%\textcolor{red}{again, proof-theorem order is strange}

%\textcolor{red}{(Need to think if we need a remark here to discuss that the cost can be further reduced if we use GMG structure.  However, the dependence would still be exponential. --XH)}

\subsubsection{A Complexity Lower Bound}
Following the same idea about the lower bound analysis for the original system in \Cref{sec:original-lower-bound}, we present the lower bound estimate in expectation, and the complexity lower bound for the expanded system \eqref{eqn:expanded-system} in this subsection.

\begin{theorem}\label{thm-lower-bound-expand}
	When $\displaystyle p_i = 1/ \widetilde{N} $ %\frac{\widetilde{A}(i,i)}{\operatorname{tr}(\widetilde{A})}$ 
	and $\omega_{i} = 1$, the RCD method applied on the expanded linear system \eqref{eqn:expanded-system} has the following lower bound estimate in expectation,
	\begin{align} \label{ine:RCD-error-expectation-A-lower-expanded}
		\mathbb{E}(\| \bm{u} - \bm{u}^{k+1} \|_A^2)  \geq  \left(  1 - \frac{\lambda_{\max}(\widetilde{D}^{-1}\widetilde{A})}{\widetilde{N}} \right)  \mathbb{E}( \| \bm{u} - \bm{u}^k \|_A^2).
	\end{align}
\end{theorem}

\begin{proof}
 Following from the proof of \Cref{thm:RCD-expanded-cnvergence}, we have
 \begin{align}\label{eqn-proof-expand-1}
 	   \mathbb{E}(\| \bm{u} - \bm{u}^{k+1} \|_A^2 \vert \bm{u}^k) & = (  \left(  I - \frac{1}{\widetilde{N}} \sum_{i=1}^{\widetilde{N}} \frac{1}{\widetilde{A}(i,i)}  \bm{e}_{i}\bm{e}_{i}^T \widetilde{A}    \right) (\widetilde{\bm{u}} - \widetilde{\bm{u}}^{k}), \widetilde{A} (\widetilde{\bm{u}} - \widetilde{\bm{u}}^{k} ))  \\
 	& = (\left(I - \frac{1}{\widetilde{N}} \widetilde{D}^{-1}\widetilde{A}\right)(\widetilde{\bm{u}} - \widetilde{\bm{u}}^{k}), \widetilde{A} (\widetilde{\bm{u}} - \widetilde{\bm{u}}^{k} )) \nonumber \\
 	& \geq \Big( 1 - \frac{\lambda_{\max}(\widetilde{D}^{-1}\widetilde{A})}{\widetilde{N}}\Big)\|\bm{u}-\bm{u}^k)\|_A^2, \nonumber 
%     \mathbb{E}(\| \bm{u} - \bm{u}^{k+1} \|_A^2 \vert \bm{u}^k) & = (  \left(  I - \frac{1}{\operatorname{tr}(\widetilde{A})} \sum_{i=1}^N  \bm{e}_{i}\bm{e}_{i}^T \widetilde{A}    \right) (\widetilde{\bm{u}} - \widetilde{\bm{u}}^{k}), \widetilde{A} (\widetilde{\bm{u}} - \widetilde{\bm{u}}^{k} ))  \\
%			& = (\left(I - \frac{1}{\operatorname{tr}(\widetilde{A})}\widetilde{A}\right)(\widetilde{\bm{u}} - \widetilde{\bm{u}}^{k}), \widetilde{A} (\widetilde{\bm{u}} - \widetilde{\bm{u}}^{k} )). \nonumber
			% & = (\bm{u} - \bm{u}^k, \bm{u} - \bm{u}^k)_A -  \frac{1}{\operatorname{tr}(\widetilde{A})}(\Pi \Pi^T A (\bm{u} - \bm{u}^k), \bm{u} - \bm{u}^k)_A  \\
 \end{align}
where the last equation holds utilizing the relationship \eqref{u-tilde-u} and this completes the proof.
\end{proof}

In what follows, we obtain the lower bound of computational complexity for the expanded system based on the above \Cref{thm-lower-bound-expand}, which is summarized in the following theorem.
\begin{theorem}\label{thm-expanded-lower}
	Under the assumption of \Cref{thm-expanded-upper}, the lower bound of complexity for the RCD method applied above is $\Omega(\operatorname{poly}(d)\epsilon^{-d}\log (\epsilon^{-1}))$ in expectation.
\end{theorem}

\begin{proof}
The convergence lower bound \eqref{ine:RCD-error-expectation-A-lower-expanded} implies that, if $\bm{u}^0 = \bm{0}$, 
\begin{align*}
	\mathbb{E}(\| \bm{u} - \bm{u}^{k} \|_A^2)  \geq  \left(  1 - \frac{\lambda_{\max}(\widetilde{D}^{-1}\widetilde{A})}{\widetilde{N}} \right)^{k} \| \bm{u}  \|_A^2.
\end{align*}
Using the inequality $\displaystyle (1-x)^k \geq \operatorname{exp}(-\frac{kx}{1-x})$, $0<x<1$, we have
\begin{align*}
	k = \Omega(\frac{\widetilde{N} - \lambda_{\max}(\widetilde{D}^{-1}\widetilde{A})}{\lambda_{\max}(\widetilde{D}^{-1}\widetilde{A})} \log \frac{1}{\epsilon}) = \Omega(\operatorname{poly}(d)(\epsilon^{-d}-1)\log\frac{1}{\epsilon}).
\end{align*}
%\begin{align*}
%k = \Omega(\frac{\operatorname{tr}(\widetilde{A}) - \lambda_{\min}(\widetilde{A})}{\lambda_{\min}(\widetilde{A})} \log \frac{1}{\epsilon}) = \Omega((\epsilon^{-d}-1)\log\frac{1}{\epsilon}).
%\end{align*}
So, after $k$ such iterations, the relative error in expectation is still greater than $\epsilon$.   Since the lower bound of each iteration of RCD is $\Omega(d)$ corresponding to accessing a sparse row, the lower bound on the overall cost of $k$ steps is $\Omega(dk)$ for the expanded system. Therefore, the complexity lower bound is 
\begin{align}\label{eqn:lower-bound-expanded}
\Omega(\operatorname{poly}(d)(\epsilon^{-d}-1)\log \frac{1}{\epsilon}),  
\end{align}
which completes the proof.
\end{proof}

In summary, 
% \cref{thm-expanded-lower} shows that, 
for the expanded system with preconditioning,   the computational complexity to achieve $\epsilon$-accuracy is still exponential in $d$.

\section{General Quantum-inspired Classical Algorithms for FEM} \label{sec:QICA-FEM} 
In \Cref{sec:RCD-FEM}, we demonstrated that the RCD  method, the current state-of-the-art quantum-inspired algorithm for sparse SPD linear systems, fails to achieve an exponential speedup in dimension when applied to the FEM. Nevertheless, the possibility remains that a yet-undiscovered quantum-inspired classical algorithm could achieve such a speedup. In this section, we address this possibility by examining the lower bounds of general quantum-inspired classical algorithms. Following the approach in \cite{mande2025lower}, we utilize communication complexity as our primary technique for deriving these bounds. Furthermore, we restrict our analysis to the class of quantum-inspired algorithms that utilize the following sampling and query access models for matrices and vectors defined in \cite{chia2022sampling, shao2022faster}.
\begin{definition}[Query access $Q(\bm{v})$ and $Q(A)$]
	For a vector $\bm{v} \in \mathbb{C}^n$, we have $Q(\bm{v})$  if for all $i \in[n]$, we can query for $v_i$. Likewise, for a matrix $A=\left(A(i, j)\right) \in \mathbb{C}^{m \times n}$, we have $Q(A)$ if for all $(i, j) \in[m] \times[n]$, we can query for $A(i,j)$.
\end{definition}
\begin{definition}[Sampling and query access to a vector $S Q(\mathbf{v})$]
	For a vector $\bm{v} \in \mathbb{C}^n$, we have $S Q(\mathbf{v})$ if we can
	\begin{itemize}
		\item[-] query for entries of $\bm{v}$ as in $Q(\bm{v})$;
		\item[-] obtain independent samples of indices $i \in[n]$, each distributed as $\operatorname{Pr}(i)=\left|\bm{v}(i)\right|^2 /\|\bm{v}\|^2$;
		\item[-] query for $\|\bm{v}\|$.
	\end{itemize}
\end{definition}
\begin{definition}[Sampling and query access to a matrix $S Q(A)$]
	For a matrix $A \in \mathbb{C}^{m \times n}$, we have $S Q(A)$ if we have $S Q\left(A(i, :)\right)$ for all $i \in[m]$ and $S Q(\mathbf{a})$ for $\mathbf{a}=\left(\left\|A(1,:)\right\|, \ldots,\left\|A(m,:)\right\|\right)$. %Here $A_{i *}$ refers to the $i$-the row of $A$.
\end{definition}
In the quantum-inspired framework, computational complexity is measured by the total number of sampling and query accesses. Therefore, deriving a lower bound for these algorithms is equivalent to finding the lower bound on the number of such accesses necessary to achieve the desired output.

%\textcolor{blue}{JA: I wonder if we should be more direct in these definitions?  I know these are the formal definitions, but to ``translate'' from this QC world's language to ours, we could be more explicit?  For instance, not use the word ``query'' to define ``query''. Also, this still seems a bit out of place here.  Where do we use these definitions precisely?}
%\textcolor{red}{XH: I did not really change the defintion because these are formal definitions. I feel we should not change those. In addition, I added a paragraph to explain why we need those.}

While the complexity analysis in \cite{mande2025lower} provides a lower bound for general matrices, which inherently applies to systems discretized from PDEs, this general bound may not be tight for the specific subset of FEM linear systems we consider here. Consequently, we provide a lower bound specifically for linear systems discretized from the PDEs described in \eqref{eqn:Ax=b}.

  % This model is essential for our lower bound analysis in \Cref{sec:QICA-FEM} based on the communication complexity analysis \cite{mande2025lower}, which again answer the questions negatively. 

%The main idea is to establish a connection between them. Then, we can use the well-known lower bounds for the Set-Disjointness problem to obtain our results.

% \subsection{A Lower Bound for General Linear Systems}
We first recall the methodology and results from \cite{mande2025lower}, as we adopt a similar constructive approach. Communication complexity, originally proposed for distributed computation in~\cite{yao1979some}, is widely utilized across numerous fields. A primary application lies in establishing lower bounds, which has been extensively studied in classical, quantum, and quantum-inspired contexts \cite{de2002quantum, rao2020communication, mande2025lower}. Following the framework in \cite{mande2025lower}, our work focuses on the \textit{coordinator model} \cite{phillips2012lower}.

In the \textit{coordinator model}, there are $k\geq 2$ players $\mathcal{P}_1,...,\mathcal{P}_k$ and a coordinator $\mathcal{C}$. Each player holds some private information, and their goal is to solve some problem (such as solving $A\bm{x} = \bm{b}$) using as little communication as possible. The communication occurs between a player and the coordinator via a 2-way private channel. The computation is in terms of rounds: at the beginning of each round, the coordinator sends a message to one of the $k$ players, and then that player sends a message back to the coordinator. In the end, the coordinator returns an answer. The communication complexity is defined to be the total number of bits sent through the channels, on a worst-case input and worst-case outcomes basis of the internal randomness of the protocol. A fundamental example is the Set-Disjointness problem, which is stated as follows.

\begin{definition}[$k$-player Set-Disjointness problem]\label{disjointness}
    For $i\in\{1,\cdots,k\}$, player $\mathcal{P}_i$ receives a bit string $T_{i} = (T_{i1}, T_{i2}, \cdots, T_{in})\in\{0,1\}^n$. Their goal is to determine if there is a $j\in\{2,3,\cdots,k\}$ such that $T_{1l} = T_{jl} = 1$ for some $l\in\{1,2,\cdots, n\}$.
\end{definition}

The following proposition details the communication complexity of $k$-player Set-Disjointness, a cornerstone for establishing lower bounds in many complexity frameworks.

\begin{proposition}[Theorem 3.3 of \cite{phillips2012lower}, Proposition 5 of \cite{mande2025lower}]\label{pro-4.2}
When $n\geq 3200 k$, for any classical protocol that succeeds with probability $1-1/k^3$ for solving the $k$-player Set-Disjointness problem, the randomized communication complexity is $\Theta(kn)$. %\textcolor{red}{(We are introducing new notation here, need to be careful what does this mean --XH)}
\end{proposition}

Next we present the connection between quantum-inspired classical algorithms and communication complexity in the following proposition, which plays an important role in proving the lower bound. 
\begin{proposition}[Theorem 10 of \cite{mande2025lower}]\label{pro-disjoint}
In the multi-player coordinator model, for each $i \in\{1, \ldots, k\}$, assume that player $\mathcal{P}_i$ holds a matrix $A^{(i)} \in \mathbb{R}^{\ell_i \times n}$ and a vector $\bm{b}^{(i)} \in \mathbb{R}^{m_i}$ with $m:=\sum_i \ell_i=\sum_i m_i$. Assume that all entries are specified by $\mathcal{O}(\log q)$ bits. Let

$$
A=\left(\begin{array}{c}
A^{(1)} \\
\vdots \\
A^{(k)}
\end{array}\right)_{m \times n}, \quad \bm{b}=\left(\begin{array}{c}
\bm{b}^{(1)} \\
\vdots \\
\bm{b}^{(k)}
\end{array}\right)_{m \times 1} .
$$
Then, we have the following results:
\begin{itemize}
    \item[-] The coordinator $\mathcal{C}$ can use $S Q(A) ~\mathcal{O}(T)$ times, using $\mathcal{O}((T+k) \log (q m n))$ bits of communication.
    \item[-] The coordinator $\mathcal{C}$ can use $S Q(\bm{b})~ \mathcal{O}(T)$ times, using $\mathcal{O}((T+k) \log (q m))$ bits of communication.
\end{itemize}
\end{proposition}

Based on \Cref{pro-disjoint}, the authors in \cite{mande2025lower} provide lower bounds for solving several different problems, such as \textit{linear regression, supervised clustering, principal component analysis, etc.} Our problem is analogous to the \textit{linear regression} problem, so we specifically focus on its corresponding lower bound, and adjust the construction accordingly.

\subsection{A Lower Bound for Quantum-inspired FEM}
The goal of this subsection is to prove a result for the linear system \eqref{eqn:Ax=b} arising from the FEM similar to 
%\Cref{thm-lower-QI}. 
\cite{mande2025lower}.
%, the proof of \Cref{thm-lower-QI} is constructive. 
Since the authors there considered general linear systems, they constructed a specific matrix $A$ and right-hand side $\bm{b}$, demonstrating that solving the corresponding linear system (or the linear regression problem) is equivalent to the $k$-player Set-Disjointness problem, thereby establishing the lower bound.

While we adopt a similar constructive approach, we do not have the same flexibility in constructing the matrix, as $A$ is fixed by the FEM discretization. Consequently, we construct a specific solution (or, equivalently, the right-hand side) to reduce our problem to a $k$-player Set-Disjointness problem. This reduction allows us to establish the lower bound for applying quantum-inspired classical algorithms to FEM linear systems. 
%The result is summarized in the following theorem.

Let us consider solving Poisson's equation with pure Neumann boundary conditions on a uniform triangulation mesh with mesh size $h$ and discretized with linear finite elements.   The sparsity of the discrete Laplacian corresponds to a graph (with $N$ nodes and $M$ edges).  Note that the nodes of the graph correspond to the vertices of the mesh, but the graph has fewer edges than the mesh. For example, in 2D, each node of the graph is incident to $4$ edges, whereas the corresponding mesh vertex is incident to $6$ edges, since the entries in the stiffness matrix corresponding to the two diagonal edges are zeros and, hence, the corresponding mesh edges are absent from the graph.   From the sparsity graph, we pick an independent set $\{v_1, \cdots, v_k\}$ and label them as the first $k$ vertices.  Then, we construct a special solution as follows,
\begin{align} \label{eqn:k-player-solution}
\bm{u}^* = \bm{e}^N_1 +
n\sum_{j=2}^k\overline{\bm{t}}_1^T \overline{\bm{t}}_j \bm{e}^N_j \in \mathbb{R}^N,
\end{align}
where the inputs to the $k$-player Set-Disjointness problem are $T_1, \ldots, T_k \in \{0,1\}^n$ with Hamming weight $\Theta(n)$ (i.e., number of nonzero $T_i$ is $\Theta(n)$) and $n \geq 3200k$.  Each player $\mathcal{P}_j$ constructs $\bm{t}_j = \sum_\ell T_{j\ell} \bm{e}^n_\ell$ and $\overline{\bm{t}}_j = \bm{t}_j/\|\bm{t}_j\|$. Here, $\bm{e}^n_{\ell}$ is the $\ell$-th unit basis vector of $\mathbb{R}^{n}$.
 
Next we use $\bm{u}^*$ to define the right-hand side. However, $A \bm{u}^*$ contains $\overline{\bm{t}}_1^T \overline{\bm{t}}_j$, which requires an information exchange between the players.  Instead, we define an equivalent least problem (i.e., linear regression).   To demonstrate our construction, first define 
\begin{align*}
\bm{x}^* = G \bm{u}^* =  \bm{g}_1 +
n\sum_{j=2}^k\overline{\bm{t}}_1^T \overline{\bm{t}}_j \bm{g}_j    \in \mathbb{R}^M,
\end{align*}
where $G \in \mathbb{R}^{M \times N}$ is the discrete gradient on the sparsity graph, or the incidence matrix between edges and nodes, and  $\bm{g}_j = G\bm{e}_j \in \mathbb{R}^M$.  Since $\{v_1, \cdots, v_k\}$ is an independent set, nonzeros components of $\bm{g}_i$ and $\bm{g}_j$ do not overlap with each other, and thus $\bm{g}_i^T \bm{g}_j = 0$ for $i\neq j$ and $i,\, j \in \{1, \dots, k\}$.  Defining $\delta_i$ as the degree of node $v_i$ of the sparsity graph, this implies that if we use $\widetilde{\bm{g}}_j \in \mathbb{R}^{\delta_i}$ to denote the vector that contains only the nonzeros components of $\bm{g}_i$ and order the edges properly, then $\bm{x}^*$ has nice block structure, 
\begin{align*}
	\bm{x}^* = 
	\begin{pmatrix}
		\widetilde{\bm{g}}_1 \\
		n  \overline{\bm{t}}_1^T \overline{\bm{t}}_2  \widetilde{\bm{g}}_2 \\
		\vdots \\
		n  \overline{\bm{t}}_1^T \overline{\bm{t}}_k  \widetilde{\bm{g}}_k \\
		\bm{0}
	\end{pmatrix}.
\end{align*}
Now we reverse-engineer and define
\begin{align} \label{eqn:define-B}
B = \mathrm{diag}\!\left(
 I_{(\delta_1)} \otimes \bm{e}^n_1,\;
I_{(\delta_2)} \otimes \overline{\bm{t}}_2,\; \ldots,\;
I_{(\delta_k)} \otimes \overline{\bm{t}}_k,\;
\tfrac{1}{\sqrt{n}} I_{(M-\sum_{j=1}^k \delta_j)} \otimes \bm{1}
\right)_{Mn \times M},
\end{align}
which satisfies $B^TB=I$, and
\begin{align} \label{eqn:define-b}
\bm{b} = 
	\begin{pmatrix}
	\widetilde{\bm{g}}_2 \otimes e^n_{1}, \\
	n  \widetilde{\bm{g}}_2 \otimes \overline{\bm{t}}_1 \\
	\vdots \\
	n  \widetilde{\bm{g}}_k \otimes \overline{\bm{t}}_1   \\
	\bm{0}
\end{pmatrix}
\in \mathbb{R}^{Mn}.
\end{align}
It is easy to check that $\bm{x}^*$ solves $\displaystyle \min_{\bm{x} \in \mathbb{R}^M} \frac{1}{2}|| B\bm{x}-b ||^2$. Since $\bm{x}^* = G \bm{u}^*$, though, we can rewrite the linear regression problem as 
\begin{align} \label{eqn:LS-prob}
 \displaystyle  \min_{\bm{u} \in \mathbb{R}^N} \frac{1}{2} \| BG \bm{u} - \bm{b} ||^2,
\end{align}
which is equivalent to solving $G^TB^TB G \bm{u}^* = G^TB^T\bm{b}$.  Since $B^TB = I$, we have  $A \bm{u}^* = \bm{f}$ where $A = h^{d-2} G^TG$ and $\bm{f} = h^{d-2}  G^T B^T \bm{b}$.  Note that $A$ is exactly a discretization of the Poisson equation on a uniform mesh with linear finite elements and pure Neumann boundary conditions. Therefore, solving the least-squares problem is equivalent to solving the Poisson equation.  Due to our special choice of the solution $\bm{u}^*$, it can then be reduced to the $k$-player Set-Disjointness problem. 

Note that the matrix $B$ and vector $\bm{b}$ both satisfy the conditions in \cref{pro-disjoint}.  Thus, we know how to sample/query them via $SQ(B)$ and $SQ(\bm{b})$.  However, to solve \eqref{eqn:LS-prob}, we also need sampling and querying of $BG$, i.e., $SQ(BG)$.  This is obtained in the following theorem.  

	\begin{theorem}[Communication complexity for special $B$, $\bm{b}$, and $G$]
	\label{thm:communication-complexity}
	%In the $k$-player coordinator model, let $\{v_1, \ldots, v_k\}$ be an independent set of a graph with $N$ nodes, $M$ edges, and maximum degree $\delta_{\max}$. Let the inputs to the $k$-player Set-Disjointness problem be $T_1, \ldots, T_k \in \{0,1\}^n$ with Hamming weight $\Theta(n)$ and	$n \geq 3200k$. Each player $P_j$ constructs	$t_j = \sum_\ell T_{j\ell}|\ell\rangle$ and $|t_j\rangle = t_j/\|t_j\|$.
	Let $G \in \mathbb{R}^{M \times N}$ be a discrete gradient matrix and let $B$ and $\bm{b}$ be defined as in \eqref{eqn:define-B} and \eqref{eqn:define-b}.   Assume that all entries are specified by $\mathcal{O}(\log q)$ bits and $\delta_{\max}$ denotes the maximal degree of the underlying sparsity graph, 
	%$|g_j\rangle$ the normalized restriction of $G|j\rangle$ to edges incident
	%to $v_j$, and $\beta_A, \beta_b > 0$ with $\beta_b = \Theta(\beta_A)$.
	%where blocks of $A$ and $b$ for $v_1$ and remaining edges are\textbf{public}, player $P_j$ holds the block of $A$ for edges incident to $v_j$ ($j = 2, \ldots, k$), and player $P_1$ holds the blocks of $b$ for edges incident to $v_j$ ($j = 2, \ldots, k$). 
	then:
	\begin{itemize}
		\item The coordinator $C$ uses $SQ(BG)$ $\mathcal{O}(T)$ times, using
		$\mathcal{O}\!\left((T + k)\,\log(qkn \delta^2_{\max})\right)$
		bits of communication.
		
		\item The coordinator $C$ uses $SQ(\bm{b})$ $\mathcal{O}(T)$ times, using
		$\mathcal{O}\!\left((T + k)\,\log(qkn\delta_{\max})\right)$
		bits of communication.
	\end{itemize}
\end{theorem}

\begin{proof}
\textbf{Cost of $SQ(\bm{b})$.}
Note that the vector $\bm{b}$ has vertical block structure. Since the first block is public and the last block is zero, both require no communication (see \cite[Remark 11]{mande2025lower}). The private blocks for $j = 2, \ldots, k$ each have size $\delta_j n$, giving a total size for all private blocks as $\mathcal{O}(kn\delta_{\max})$.  Again, each entry contains $\mathcal{O}(\log q)$ bits. By \cref{pro-disjoint}, $\mathcal{O}(T)$ uses of $SQ(\bm{b})$ cost $\mathcal{O}((T+k)\log(qkn\delta_{\max}))$ bits.

\textbf{Cost of $SQ(BG)$.} We proceed in two steps.

\textit{Step 1: Simulate $SQ(B)$ via communication.}
The matrix $B$ has block diagonal structure, where the first  and last diagonal blocks are public, and no communication is needed. The private blocks for $j = 2, \ldots, k$ each have size $\delta_j n \times \delta_j$, giving total size  of all private blocks as $O(kn\delta_{\max}^2)$. By \cref{pro-disjoint}, $\mathcal{O}(T)$ uses of $SQ(B)$ cost $\mathcal{O}((T+k)\log(qkn\delta_{\max}^2))$ bits.

\textit{Step 2: Construct $SQ(BG)$ from $SQ(B)$.} Since $B$ is block diagonal with each row having exactly one nonzero entry $B_{ij}$, each row $i$ of $BG$ has exactly two nonzero entries, which are $B_{ij}$ and $-B_{ij}$.
%\[
%(AG)_{i,e_{j(i)}^+} = A_{i,j(i)}, \qquad (AG)_{i,e_{j(i)}^-} = -A_{i,j(i)},
%\]
%where $e_{j(i)}^+$ and $e_{j(i)}^-$ are the two endpoints of edge $j(i)$.
Therefore, $SQ(BG)$ can be simulated directly from $SQ(B)$ with $\mathcal{O}(1)$ overhead per access:
Querying $(BG)(i,j)$ uses one query to $Q(B)$ and the public $G$, with the overhead cost $\mathcal{O}(1)$;
Then, computing $\| (BG)(i,:) \| =
	\sqrt{2} \| B(i,:) \|$ requires one query to $Q(B)$ with overhead cost $\mathcal{O}(1)$;
	Next, to get a row sample, since $\|(BG)(i,*)\|^2 = 2 \|B(i,*)\|^2$, the row distribution of $BG$ equals that of $B$. Therefore, we sample directly from $SQ(B)$ with overhead cost $\mathcal{O}(1)$; 
	Finally, for within-row sampling, since the two nonzero entries have equal magnitude, we sample uniformly using the public $G$ with overhead cost $\mathcal{O}(1)$.
    
Overall, $\mathcal{O}(T)$ uses of $SQ(BG)$ require only  $\mathcal{O}(T)$ uses of $SQ(B)$. Combining with \textit{Step~1}, we see that  $\mathcal{O}\left((T+k)\,\delta_{\max}\,\log(qkn)\right) $ bits suffice for $\mathcal{O}(T)$ uses of $SQ(BG)$.
Collecting all the results completes the proof. 
\end{proof}

One immediate consequence of the communication complexity defined in \Cref{thm:communication-complexity}, is getting the lower bound on the applications of $SQ(BG)$ and $SQ(\bm{b})$ required for solving the least-squares problem \eqref{eqn:LS-prob} given $B$ in \eqref{eqn:define-B} and $\bm{b}$ in \eqref{eqn:define-b}.  This result is summarized in the following corollary.

\begin{corollary}\label{coro:QIC-LS-lower-bound}
Any quantum-inspired algorithm that solves the least-squares problem, \eqref{eqn:LS-prob}, by making $T$ applications of $SQ(BG)$ and $SQ(\bm{b})$ must satisfy:
\begin{equation*}
T \;\geq\; \Omega\left(\frac{k^2}{\log (k \delta_{\max}^2)}\right).
\end{equation*}
\end{corollary}

\begin{proof}
We first show that the solution of the least-squares problem, \eqref{eqn:LS-prob}, that we constructed is equivalent to the Set-Disjointness problem.   Recall that 
\begin{align*}
\bm{u}^* = \bm{e}^N_1 +
n\sum_{j=2}^k\overline{\bm{t}}_1^T \overline{\bm{t}}_j \bm{e}^N_j.
\end{align*}
If there is no intersection for the $k$-player Set-Disjointness problem, i.e., no $j \in \{2, 3, \cdot, k\}$ such that $T_{1l} = T_{jl}$ (see \cref{disjointness}), then $\bm{u}^* = \bm{e}^{N}_1$.  Sampling $\bm{u}^*$ returns index $1$ with probability $1$.  If there is an intersection with one common index, then there is a $j \in \{ 2, \cdots, k \}$ such that
\begin{align*}
	\bm{u}^* = \bm{e}^N_1 +
	n \overline{\bm{t}}_1^T \overline{\bm{t}}_j \bm{e}^N_j.
\end{align*}
Note that $n \overline{\bm{t}}_1^T \overline{\bm{t}}_j = \Theta(1)$ since we assume the Hamming weight is $\Theta(n)$.  Then, with constant probability, we get index $j$ by sampling from the distribution defined by $\bm{u}^*$.  Thus, if we can sample from $\bm{u}^*$, we can solve the $k$-player Set-Disjointness problem using a constant number of such samples. 
		
This connection allows us to invoke \Cref{pro-4.2}.  By \Cref{thm:communication-complexity}, a quantum-inspired algorithm of complexity $T$ implies a communication protocol of complexity $\mathcal{O}((T+k)\log(qkn \delta_{\max}^2))$. By \Cref{pro-4.2}, the $k$-player Set-Disjointness problem with $n \geq 3200k$ requires $\Theta(kn)$ communication. Therefore, choosing $n = \Theta(k)$, we have
\begin{align*}
\mathcal{O}\left((T+k)\,\log(qk \delta^2_{\max})\right) \;\geq\; \Theta(k^2) \implies T \;\geq\; \Omega\left(\frac{k^2}{\log (k \delta_{\max}^2)}\right),
\end{align*}
which finishes the proof. 
\end{proof}

Finally, tying this all together, solving the least-squares problem, \eqref{eqn:LS-prob}, is equivalent to solving an  instance of the Poisson equation using linear finite elements.  Therefore, we establish the following lower bound for any quantum-inspired algorithm for solving Poisson's equation with finite elements.

%\textcolor{blue}{JA: this result is worded funny.  I'm not sure it's saying what you want.}
\begin{theorem}[Lower bound for solving quantum-inspired FEM]
	\label{thm:laplacian_lower_bound}
Consider the Laplace equation with pure Neumann boundary conditions discretized with linear finite-elements on a uniform mesh, then any quantum-inspired classical algorithm that solves the resulting discrete linear system must make at least $\displaystyle \Omega\left( \frac{\epsilon^{-2d}}{d \log (\epsilon^{-1})} \right)$ calls to $SQ(BG)$ and $SQ(\bm{b})$.
\end{theorem}
\begin{proof}
From \Cref{coro:QIC-LS-lower-bound}, we know that 
$T \geq \Omega\left(\frac{k^2}{\log (k \delta_{\max}^2)}\right)$.  For the sparsity graph, we have $\delta_{\max} = \Theta(d)$.  Recall that, by construction, $k$ is the size of the maximal independent set.  Choosing different $k$ corresponds to solving different Poisson problems (see our construction \eqref{eqn:LS-prob}).  Thus, to make the problem difficult for quantum-inspired classical algorithms, we choose $k$ to be as large as possible. It is well-known that a greedy algorithm for finding an approximate maximal independent set can achieve $k = \Theta(N/\delta_{\max})$.  Thus, with $N = \Theta(\epsilon^{-d})$, we have
\begin{align*}
 T \geq \Omega \left( \frac{\epsilon^{-2d}}{d \log (\epsilon^{-1})}  \right). 
\end{align*}
This completes the proof. 
\end{proof}
Again, we see that the lower bound on complexity is exponential in dimension.  Thus, no quantum-inspired could, based on this communication complexity framework, ever overcome the ``\textit{curse of dimensionality}''.

\begin{remark}
\Cref{thm:laplacian_lower_bound} shows that there exists a specific linear system, obtained from discretizing a Laplace equation, for which any quantum-inspired classical algorithm requires complexity at least exponential in the dimension. This lower bound suffices to demonstrate the gap between quantum-inspired classical algorithms and quantum algorithms for solving PDEs. Moreover, \Cref{thm:laplacian_lower_bound} also covers any quantum-inspired classical algorithm that uses preconditioning, since such algorithms form a special subset of the general quantum-inspired classical algorithms considered here.
\end{remark}

\section{Conclusions}\label{sec:conclusions}
In summary, using the RCD method, a state-of-the-art\newline quantum-inspired classical algorithm for solving SPD linear systems, we establish upper and lower complexity bounds for both the original (unpreconditioned) and expanded (preconditioned) systems arising from FEM discretizations of the high-dimensional Poisson equation. Because both the upper and lower bounds depend exponentially on the dimension $d$, we conclude that achieving an exponential speedup with respect to $d$ is impossible when applying this quantum-inspired algorithm to the FEM. This directly resolves \Cref{ques1}, posed at the outset. 

Furthermore, to address \Cref{ques2}, we carry out a lower bound analysis for general quantum-inspired classical algorithms based on quantum communication complexity. As shown in \Cref{thm:laplacian_lower_bound}, the resulting lower bound again scales exponentially with the dimension $d$. This rigorously establishes that such classical methods cannot achieve an exponential speedup in dimension, rendering them non-competitive with true quantum PDE solvers in high-dimensional settings.

Our lower bound analysis for general quantum-inspired classical algorithms is established primarily through the framework of quantum communication complexity and its associated protocols. Looking forward, a promising research direction is to refine these techniques to derive tighter, ideally optimal, lower bounds for these tasks. From an algorithmic perspective, we also plan to extend quantum-inspired classical frameworks to more advanced discretization schemes. Specifically, we aim to investigate their efficiency when coupled with the mixed finite-element method for solving high-dimensional coupled PDEs, exploring whether similar computational barriers or potential speedups arise in the saddle-point setting.

%\section*{Acknowledgments}

\bibliographystyle{siamplain}
\bibliography{ref}
\end{document}